\documentclass[11pt]{elsarticle}

\usepackage[latin1]{inputenc}
\usepackage[notcite, notref]{}
\usepackage{amsfonts,amsmath,amstext,amsbsy,amsthm,amscd,amssymb}
\usepackage{multirow}
\usepackage{graphicx}
\usepackage{float}
\usepackage{placeins}
\usepackage{bbm}
\usepackage{color}

\def\R{\mathbb{R}}

\def\P{\mathbb{P}}

\def\Pzero{\P_{\theta_0}}
\def\Pz{\P_{\theta_0}}
\def\Sk{S_k^{\alpha_0,\beta_0}}
\def\Zk{Z_k^{\alpha_0,\beta_0}}
\def\Nk{N_k(X,\alpha)}
\def\Nkd{N_k(x_{\alpha_0},\alpha)}
\def\Nkzero{N_k(X,\alpha_0)}
\def\Nkz{N_k(X,\alpha_0)}

\def\Sig{\Sigma^{-1}(\beta,X_{t_{k-1}})}
\def\Sigzero{\Sigma^{-1}(\beta_0,X_{t_{k-1}})}
\def\somn{\som{k=1}{n}}

\makeatletter

\@addtoreset{equation}{section}
\makeatother

\newtheorem{theo}{Theorem}[section]

\newtheorem{prop}{Proposition}[section]
\newtheorem{lemma}{Lemma}[section]
\newtheorem{corol}{Corollary}[section]

\newtheorem{rem}{\it Remark}[section]

\newcommand{\som}[2]{\displaystyle{\sum_{#1}^{#2}}}

\newcommand{\trans}[1]{\,{\vphantom{#1}}^{t}\!{#1}}

\newcommand{\tend}[2]{\underset{{#1}\rightarrow {#2}}{\longrightarrow}} 
\newcommand{\abs}[1]{\lvert #1 \rvert}

\newcommand{\integ}[2]{\displaystyle{\int_{#1}^{#2}}}

\newcommand{\norm}[1]{\lVert #1 \rVert}
\newcommand{\deriv}[2]{\frac{\partial {#1}}{\partial {#2}}}
\newcommand{\E}[1]{\mathbb{E}\left[#1\right]}
\newcommand{\super}[1]{\underset{#1}{sup}}
\newcommand{\Ec}[2]{\E{#1 | \mathcal{F}_{t_{#2}}}}
\newcommand{\x}[1]{x_\alpha(t_{#1})}
\newcommand{\xzero}[1]{x_{\alpha_0}(t_{#1})}
\newcommand{\xz}[1]{x_{\alpha_0}(t_{#1})}
\newcommand{\fy}[2]{\Phi_\alpha(t_{#1},t_{#2})}
\newcommand{\fyzero}[2]{\Phi_{\alpha_0}(t_{#1},t_{#2})}
\newcommand{\fyz}[2]{\Phi_{\alpha_0}(t_{#1},t_{#2})}
\newcommand{\Reste}[1]{R^{2,\epsilon}_{\theta_0}(t_{#1})}
\newcommand{\Restun}[1]{R^{1,\epsilon}_{\theta_0}(t_{#1})}
\newcommand{\Ft}[1]{\mathcal{F}_{t_{#1}}}
\newcommand{\supk}[1]{\super{k\in\{1,..,n\}}\norm{#1}}
\newcommand{\Sigdo}[1]{\Sigma^{-1}(\beta_0,x_{\alpha_0}(t_{#1}))}
\newcommand{\Sigdet}[1]{\Sigma^{-1}(\beta,x_{\alpha_0}(t_{#1}))}
\newcommand{\dbx}[1]{\deriv{b}{x}(\alpha,x_{\alpha}(t_{#1}))}
\newcommand{\dbxz}[1]{\deriv{b}{x}(\alpha_0,x_{\alpha_0}(t_{#1}))}
\begin{document}
 
\begin{frontmatter}
\date{}
\title{{\bf Parametric inference for  discretely observed multidimensional diffusions with small diffusion coefficient}}
\author[lab1,lab2]{Romain GUY\corref{cor}}
\ead{romain.guy@jouy.inra.fr}
\author[lab1,lab2]{Catherine Lar\'edo}
\author[lab1]{Elisabeta Vergu}
\cortext[cor]{Corresponding Author}
\address[lab1]{UR 341 Math\'ematiques et Informatique Appliqu\'ees, INRA, Jouy-en-Josas, France}
\address[lab2]{UMR 7599 Laboratoire de Probabilit\'es et Mod\`eles al\'eatoires, Universit\'e Denis Diderot Paris 7 and CNRS, Paris, France}

\begin{abstract}
	We consider a multidimensional diffusion $X$ with drift coefficient $b(\alpha,X_t)$ and diffusion 
	coefficient $\epsilon \sigma( \beta,X_t)$. The diffusion sample 
	path is discretely observed at times  $t_k=k\Delta$ for $k=1..n$ on a 
	fixed interval $[0,T]$. We study minimum contrast estimators derived from the Gaussian process approximating 
	$X$ for small $\epsilon$.
	We obtain  consistent and asymptotically normal estimators 
	of  $\alpha$ for fixed  $\Delta$  and $\epsilon\rightarrow 0$ and  of 
	$(\alpha,\beta)$ for $\Delta \rightarrow 0$ and $\epsilon  \rightarrow 0$ without any condition linking $\epsilon$ and $\Delta$.
	We compare the estimators obtained with various methods and for various
	magnitudes of $\Delta$ and $\epsilon$ based on simulation studies. Finally, we investigate the
 	interest of using such methods in an epidemiological framework.
\end{abstract}

%
\end{frontmatter}
\section{Introduction}

In this study we focus on the parametric  inference in the drift coefficient
$b(\alpha,X^\epsilon_t)$ and in the diffusion coefficient $\epsilon \sigma( \beta,X^\epsilon_t)$ of a multidimensional diffusion model $\left( X^\epsilon_t \right)_{t\geq 0}$ with small diffusion coefficient, 
when it is observed at discrete times on a fixed time interval in the asymptotics $\epsilon \rightarrow 0$. This asymptotics has been widely studied and has proved fruitful in applied problems, see e.g.\cite{frei84}. Our interest in considering this kind of diffusions is motivated by the fact that they are natural approximations of epidemic processes. Indeed, the classical stochastic $SIR$ model in a closed population, describing variations over time in Susceptible ($S$), Infectious ($I$) and Removed ($R$) infividuals, is a bi-dimensional continuous-time Markovian jump process. The population size ($N$) based normalization of this process asymptotically leads to an ODE system. Before passing to the limit, the forward Kolmogorov diffusion equation allows describing the epidemic dynamics through a bidimensional diffusion, with diffusion coefficient proportional to $1/\sqrt{N}$. Moreover, epidemics are discretely observed and therefore we are interested in the statistical setting  defined by discrete data sampled at times $t_k=k \Delta$ on a fixed interval $[0,T]$ with $T=n \Delta$. The number of data points is $n$ and $\Delta$, the sampling interval, is not necessarily small.

Historically, statistics for diffusions were developed for continuously observed processes 
leading to  explicit formulations of the likelihood (\cite{kut84}, \cite{lip01}). In this context, two asymptotics exist for estimating $\alpha$ for a diffusion continuously observed on a time interval $[0,T]$: $T\rightarrow \infty$ for recurrent diffusions and $T$ fixed and the diffusion coefficient tends to $0$. In practice, however, observations are not continuous but partial, with various mechanisms underlying the missingness, which leads to intractable likelihoods. One classical case consists in sample paths discretely observed with a sampling interval $\Delta$. 
This  adds another asymptotic framework $\Delta \rightarrow 0$ and raises  the question of estimating  parameters in the diffusion coefficient (see \cite{gen93}, \cite{yos92} for $T$ fixed and \cite{hans98}, \cite{kes00}, \cite{uch12} for $T \rightarrow \infty$).

Since nineties, statistical methods associated to discrete data have been developed in the asymptotics of small diffusion coefficient (e.g. \cite{lar90}, \cite{gen90}, \cite{uch04}). Considering a discretely observed diffusion on $\R$  with constant ($=\epsilon$) diffusion coefficient, Genon-Catalot (1990) obtained, using the Gaussian approximating process \cite{aze82}, a consistent and $\epsilon^{-1}$-normal and efficient estimator of  $\alpha$ under the condition $ \{\epsilon \rightarrow 0, \Delta \rightarrow 0, \epsilon/ \sqrt{\Delta} = O(1)\}$. The author additionally proved that this estimator possessed good properties also for $\Delta$ fixed. Uchida \cite{uch04} obtained similar results using approximate martingale estimating equations. Then, S\o rensen  \cite{sor00} obtained, as $\epsilon \rightarrow 0$, consistent and $\epsilon^{-1}$-normal estimators of a parameter $\theta$ present in both the drift and diffusion coefficient, with no assumption on $\Delta$, but
under additional conditions not verified in the case of distinct parameters in the drift and diffusion coefficient. For this latter case, S\o rensen and Uchida \cite{sor03} obtained consistent and $\epsilon^{-1}$-normal estimators of $\alpha$ and consistent and $\sqrt{n}$-normal estimators of $\beta$ under the condition 
$\Delta/\epsilon \rightarrow 0$ and $\sqrt{\Delta}/\epsilon$ bounded. This result was later extended by Gloter and S\o rensen \cite{glo09} to the case where $\epsilon^{-1}\Delta^{\rho}$ is bounded for some $\rho >0$. Their results rely on a class of contrast processes based on the expansion of the infinitesimal generator of the diffusion, the order of the expansion being driven by the respective magnitude of $\epsilon $ and $\Delta$ and requiring this knowledge (value of $\rho$), which might be a drawback when applying the method. Moreover, this contrast becomes difficult to handle for  values of $\Delta$ that are not very small with respect to $\epsilon$.

To overcome this drawback, we consider a simple contrast based on the Gaussian approximation of the diffusion process $X^{\epsilon}$ (\cite{aze82}, \cite{frei84}). Contrary to Gloter and S\o rensen \cite{glo09}, our contrast has generic formulation, regardless to the ratio between $\Delta$ and $\epsilon$. Thus,
the standard balance condition between $\epsilon$ and $\Delta$ of  previous works is here removed. Our study extends  the results of \cite{gen90} to the case of multidimensional diffusion processes with parameters in both the 
drift and diffusion coefficient. We consider successively the cases $\Delta$ fixed and $\Delta
\rightarrow 0$. We obtain consistent and $\epsilon^{-1}$-normal estimators of $\alpha$ (when $\beta$ is unknown or equal to a known function of $\alpha$) for fixed $\Delta$. For high frequency data, we obtain results similar to \cite{glo09}, but without any assumption on $\epsilon$ with respect to $\Delta$. The estimators obtained are analytically calculated on a simple example, the Cox-Ingersoll-Ross (CIR) model. Finally, they are compared based on simulation studies in the case of a financial two-factor model \cite{lon95} and of the epidemic $SIR$ model \cite{die00}, for various magnitudes of $\Delta$ and $\epsilon$.

The paper is structured as follows. After an introduction, Section 2 contains the notations and 
preliminary results on the stochastic Taylor expansion of the diffusion. Sections 3 and 4, which constitute the core of the paper, present analytical results, both in terms of contrast functions and estimators properties. We investigate in Section 3 the inference when $\Delta$ fixed and $\epsilon \rightarrow 0$ in three contexts depending on whether the parameter $\beta$ in the diffusion coefficient is unknown, equal to a known function of $\alpha$ (with the special case $\beta=\alpha$) or whether the diffusion coefficient is multiplicative. Section 4 is devoted to the case $\Delta \rightarrow 0$. Results are applied in Section 5 to the CIR model. Moreover, the different estimators obtained are compared based on numerical simulations to the minimum contrast estimator of Gloter and S\o rensen \cite{glo09}, mainly in the context of epidemic data.

\section{Notations and preliminary results}

Let us consider on a probability space $\left( \Omega, {\cal A},({\cal A}_t)_{t\geq 0}, \mathbf{P} \right)$ the p-dimensional diffusion process  satisfying the stochastic differential equation
\begin{equation}\label{model}
 \left\{ \begin{array}{l}
   dX^\epsilon_t=b(\alpha,X^\epsilon_t)dt+\epsilon \sigma(\beta,X^\epsilon_t)dBt\\
   X^\epsilon_0=x_0,
 \end{array} \right.
\end{equation}
where $x_0 \in \R^p$ is prescribed, $\epsilon > 0$, $\theta = (\alpha,\beta)$ are unknown  multi-dimensional parameters,  $b(\alpha, x)$ is a vector in $\R^p$, $\sigma(\beta, x)$ is a $p \times p$ matrix and $(B_t)_{t \geq 0}$ is a $p$-dimensional Brownian motion defined on $\left( \Omega, {\cal A}\right)$.

Throughout the paper we use the convention that objects are indexed by $\theta$ when there is a dependence on both $\alpha$ and $\beta$ and by $\alpha$ or $\beta$ alone otherwise. Let us denote by  $M_p(\R)$ the set of $p\times p$ matrices, and by $\trans{M}$, $Tr(M)$ and $det(M)$ respectively the transpose, trace and determinant of a matrix $M$.\\
We denote the partial derivatives of a function $f(\alpha,x)$ in $(\alpha_0,x_0)$ by $\deriv{f}{\alpha}(\alpha_0,x_0)$ and $\deriv{f}{x}(\alpha_0,x_0)$. Moreover, if $x=x(\alpha,t)$ the derivative of the function\\ $\alpha\rightarrow f(\alpha,x(\alpha,t))$ in $\alpha_0$ will be denoted by \\$\deriv{f(\alpha,x(\alpha,t))}{\alpha}(\alpha_0)=\deriv{f}{\alpha}(\alpha_0,x(\alpha_0,t))+\deriv{f}{x}(\alpha_0,x(\alpha_0,t))\deriv{x}{\alpha}(\alpha_0,t)$.\\
We set  
\begin{equation}
\Sigma(\beta,x)=\sigma(\beta,x)\trans{\sigma(\beta,x)}.
\end{equation}
In what follows, we assume that  ${\cal A}= sup({\cal A}_t ,\; {t\geq 0})$, $({\cal A}_t)_{t\geq 0}$  is right-continuous and\\
{\bf(H1)} $
\left\{
{\begin{array}{l}
(i)\; \exists U ,  \mbox{open set of } \R^p  \mbox{ such that, for small enough }
\epsilon, \; \forall t \in [0,T],\;
 X^\epsilon_t  \in U \\
(ii)\; b(\alpha,\cdot) \in C^2(U,\R^p), 
\sigma(\beta,\cdot) \in C^2(U,{\cal M}_p)\\
(iii) \;  \exists K >0,  \; \norm{b(\alpha,x)-b(\alpha,y)}^2 
+\norm{\sigma(\beta,x)-\sigma(\beta,y)}^2 \leq K \norm{x-y} ^2\\ 
\end{array}}\right.$\\
{\bf(H2)} $
\forall  x \in U ,\; \Sigma(\beta,x)\mbox{ is invertible}$\\
Assumptions {\bf(H1)} and  {\bf(H2)} ensure existence and uniqueness of a strong solution of (\ref{model}), with infinite explosion time (see e.g. \cite{ike89}).

\subsection{Results on the ordinary differential equation}

Consider the solution $x_\alpha(t)$ of the ODE associated with $\epsilon=0$ in (\ref{model})
\begin{equation}
 \left\{ \begin{array}{l}\label{deter}
   dx_\alpha(t)=b(\alpha,x_\alpha(t))dt\\
   x_\alpha(0)=x_0\in \R^p.
  \end{array} \right.
\end{equation}
Under {\bf(H1)}, this solution is well defined, unique and belongs to $ C^2 (U,\R^p)$. 
Let us consider the matrix $\Phi_\alpha(\cdot,t_0) \in {\cal M}_p$, solution of 
\begin{equation}
 \left\{ \begin{array}{l}
   \frac{d\Phi_\alpha}{dt}(t,t_0)=\deriv{b}{x}(\alpha,x_\alpha(t))\Phi_\alpha(t,t_0) 
\label{eqphi}\\
\Phi_\alpha(t_0,t_0)=I_p.
  \end{array} \right.
\end{equation}

Under {\bf(H1)}, it is well known (see e.g. \cite{car71}) that, for $ t_0\in [0,T]$,  
$\Phi_\alpha(\cdot,t_0)$ is twice continuously differentiable on $[0,T]$ and satisfies the semi-group property
\begin{equation}\label{semigroup}
\forall (t_0,t_1,t_2)\in [0,T]^3 , \;\; \Phi_\alpha(t_2,t_0)=\Phi_\alpha(t_2,t_1)\Phi_\alpha(t_1,t_0).
\end{equation}

A consequence of (\ref{semigroup}) is that the matrix $\Phi_\alpha(t_1,t_0)$ is invertible with inverse $\Phi_\alpha(t_0,t_1)$.

\subsection{Taylor Stochastic expansion of the diffusion $ (X_t^{\epsilon})$}

We use in the sequel some  known results for small perturbations of dynamical systems 
(see \cite{frei84,aze82}). The family of diffusion processes  $(X_t^{\epsilon},\;t \in[0,T])$ solution of (\ref{model}) satisfies the following theorem.
\begin{theo} \label{theo:Azen}
Under  {\bf(H1)}, 
\begin{equation}\label{taylor_stochastique}
 X^\epsilon_t=x_\alpha(t) +\epsilon g_{\theta}(t) +\epsilon ^2 R^{2,\epsilon}_{\theta} (t)
\mbox{ with } \underset{t \in [0,T]}{sup} \{ \norm{\epsilon R^{2,\epsilon}_{\theta}(t)}\} 
 \tend{\epsilon}{0}0  \mbox { in probability}
\end{equation}
and with   $x_\alpha(\cdot)$ defined in (\ref{deter}) and
$g_{\theta}(t)$  satisfying 
\begin{equation}\label{eqg}
dg_{\theta}(t)=\deriv{b}{x}(\alpha,x_\alpha(t))g_{\theta}(t)dt+\sigma(\beta,x_\alpha(t))dB_t, with \;\; g_{\theta}(0)=0.
\end{equation}
\end{theo}

\begin{rem} We use also in sequel the Taylor expansion of order 1
\begin{equation}\label{tsto1}
X^\epsilon_{t}=x_{\alpha}(t)+\epsilon R^{1,\epsilon}_{\theta}(t) \mbox{ with }
\underset{t\in[0,T]}{sup} \norm{\epsilon R^{1,\epsilon}_{\theta}(t)}
\tend{\epsilon}{0}0 \mbox{ in probability}.
 \end{equation}
\end{rem}

\begin{corol}\label{galpha}
Under {\bf(H1)}, the process  $g_{\theta}(.)$ is the continuous Gaussian martingale on $[0,T]$
defined, using (\ref{eqphi}), by
	\begin{equation}\label{defg}
       g_{\theta}(t)=\integ{0}{t}\Phi_\alpha(t,s)\sigma(\beta,x_\alpha(s))dB_s. 
      \end{equation}
\end{corol}

\begin{proof}
Using (\ref{semigroup}), the matrix $\Phi_{\alpha}(t,0)$ is invertible with inverse $\Phi_{\alpha}(0,t)$. The process $C(t)$ defined by $ g_{\theta}(t)=  \Phi_{\alpha}(t,0)C(t)$  satisfies, using  (\ref{eqg}),
$dC(t)=\Phi_{\alpha}(0,t)\sigma(\beta,x_{\alpha}(t))dB_t $ and $C(0)=0$. Thus, applying (\ref{semigroup}) yields  (\ref{defg}).  
\end{proof}

\begin{corol}\label{Rteps}
Assume {\bf(H1)}. If, moreover, $b(\alpha,.)$ and $\sigma(\beta,.)$ have uniformly bounded derivatives  on $U$, then there exist constants only depending on $T$ and $\theta $ such that\\
(i) $\forall t \in[0,T]$, $\E(\norm{R^{2,\epsilon}_{\theta}(t)}^{2}) < C_1 $,\\
(ii) $\forall t \in[0,T]$, as $h \rightarrow 0$, $\E{\norm{R^{2,\epsilon}_{\theta}(t+h)-R^{2,\epsilon}_{\theta}(t)}^{2}}< C_2 h$.
\end{corol}
The result \ref{Rteps}-(i) is given in \cite{frei84} (Theorem 2.2 p.56), proof of \ref{Rteps}-(ii) is given in Appendix \ref{ProofRteps}.\\ 
An important consequence of Corollary \ref{galpha} is the following lemma on the Gaussian process $g_{\theta}$. Let us define
\begin{equation}\label{eq:Zk}
Z^{\theta}_{k}=\frac{1}{\sqrt{\Delta}}\int^{t_k}_{t_{k-1}} \Phi_{\alpha} (t_{k},s) \sigma (\beta, x_{\alpha}(s))\; dB_s,
\end{equation}
\begin{equation}\label{S_k+1} 
  S^{\alpha,\beta}_{k}=\frac{1}{\Delta}\integ{t_{k-1}}{t_{k}}\Phi_\alpha(t_{k},s)\Sigma(\beta,x_\alpha(s))
\trans{\Phi_\alpha(t_{k},s)}ds.
\end{equation}
\begin{lemma}\label{lemma:gZ}
Under {\bf (H1)}, the random variables $g_{\theta}(t_{k})$  verify, for  $t_k=k\Delta, \; k=1,\dots, n$, 
	\begin{equation} \label{relg}
	g_{\theta}(t_{k})=\Phi_\alpha(t_{k},t_{k-1})g_{\theta}(t_{k-1})+\sqrt{\Delta}Z^{\theta}_{k},
	\end{equation}
where $(Z^{\theta}_{k})_{1\leq k\leq n}$ defined in (\ref{eq:Zk}) is a sequence of $\R^p$- dimensional independent centered Gaussian random variables, $ \mathcal{A}_{t_{k}}$- measurable and with the covariance matrix $S_k^{\alpha,\beta}$
\end{lemma}
\begin{proof}
Using (\ref{defg}) and the semi-group property of $\Phi_{\alpha}(t,s)$ yields\\
  $g_{\theta}(t_{k})= \Phi_{\alpha}(t_{k},t_{k-1}) g_{\theta}(t_{k-1})+ 
  \int^{t_k}_{t_{k-1}}\Phi_{\alpha} (t_{k},s) \sigma (\beta, x_{\alpha}(s))\; dB_s$. 
The proof is achieved by identifying $Z_k^\theta$ in this relation.
\end{proof}
Note that {\bf (H1)} and {\bf (H2)} ensure that $S_k^{\alpha,\beta}$ is positive definite matrix.

\subsection{Statistical framework}

Let  $\mathbf{C}=C ([0,T],\R^p)$ denote the space of continuous functions defined on $[0,T]$ with values in $\R^p$ endowed with the 
uniform convergence topology, ${\cal C}$  the $\sigma$-algebra of the Borel sets, ($X_t$) the canonical coordinates of $(\mathbf{C},{\cal C})$ 
and  ${\cal F}_{t}= \sigma(X_s, 0\leq s \leq T)$. Finally, let $\mathbb{P}^\epsilon_{\theta} = \mathbb{P}^\epsilon_{\alpha,\beta}$ be the distribution on $(\mathbf{C},{\cal C})$ of the diffusion process solution of (\ref{model}).\\
From now on,  let $\theta_0=(\alpha_0,\beta_0) \in \Theta$ be the true value of 
the parameter. We assume\\
\\
{\bf(S1)} $(\alpha,\beta) \in K_a \times K_b= \Theta$ 
with $ K_a, K_b$  compacts sets  of  $\R^a,\; \R^b$; $\theta_0 \in 
\mathring{\Theta}$\\
{\bf(S2)} (H1)-(H2)  hold  for all $(\alpha,\beta) \in 
\Theta$ with constant $K$ not depending on $\theta$\\
{\bf(S3)} The function $b(\alpha,x)$ is $ C^3(K_a \times U,\R^p)$ and 
$\sigma(\beta,x) \in C^2(K_b\times U,{\cal M}_p)$\\
{\bf(S4)} $\Delta \rightarrow 0$: $ \alpha \neq \alpha' 
\Rightarrow b(\alpha,x_\alpha(\cdot))\neq
b(\alpha',x_{\alpha'}(\cdot))$\\
{\bf(S4')} $\Delta$ fixed: $\alpha \neq \alpha' \Rightarrow \{ \exists k, 
\quad  1 \leq k\leq n, \quad x_\alpha(t_k)\neq
x_{\alpha'}(t_k)\}$\\
{\bf(S5)}  $\beta \neq \beta' \Rightarrow  \Sigma(\beta,x_{\alpha_0}(\cdot)) \neq \Sigma(\beta',x_{\alpha_0}(\cdot))$.

Assumptions {\bf(S1)-(S3)} are classical for the inference for diffusion processes. The differentiability in {\bf(S3)} comes from the regularity conditions required on $\alpha \rightarrow \Phi_{\alpha}(t,s)$. Indeed, {\bf(S3)} on $b(\alpha,x)$ ensures  that $\Phi_{\alpha}(t,t_0)$ belongs to $ C^2(K_a\times [0,T]^2, {\cal M}_p)$ (see Appendix \ref{rem:Phi} for the proof). {\bf(S4)} is the usual identifiability assumption for a continuously observed diffusion on $[0,T]$. Note that {\bf(S4)} ensures that, for $\Delta$ small enough, {\bf(S4')} holds.

For a sample path $y(.) \in C([0,T],\R^p)$, let us define the quantity  depending on $x_{\alpha}(.)$, $\Phi_{\alpha}(.,.)$ and on the discrete sampling
$(y_{t_k}, k=1, \dots,n)$,
 \begin{equation}\label{def:Nk}
 N_{k}(y,\alpha)= y(t_{k})- x_\alpha(t_{k})-\Phi_\alpha(t_{k},t_{k-1})(y(t_{k-1})-x_\alpha(t_{k-1})).
 \end{equation}
Note that $N_k(x_\alpha,\alpha)=0$. Let us also define the Gaussian process $(Y^\epsilon_t) _{t\in[0,T]}\in \mathbf{C}$,

 \begin{equation*} \label{definition_Y}
Y^\epsilon_t = x_\alpha(t)+ \epsilon g_{\theta}(t).
\end{equation*}
Using (\ref{relg}) and (\ref{def:Nk}), we can express the random variables $Z_{k}^{\theta}$ using  $Y^\epsilon_t$,
\begin{equation}\label{eq:Zkbis}
	Z_{k}^{\theta}= \frac{Y^\epsilon_{t_{k}}-x_\alpha(t_{k})}{\epsilon\sqrt{\Delta}} -
\Phi_\alpha(t_{k},t_{k-1})\frac{Y^\epsilon_{t_{k-1}}-x_\alpha(t_{k-1})}{\epsilon\sqrt{\Delta}} = N_k( \frac{Y^\epsilon_{.}-x_\alpha(.)}{\epsilon\sqrt{\Delta}},\alpha).
\end{equation}

Then, the n-sample $(Y_{t_k}, k=1,\dots,n) $ has an explicit loglikelihood $l(\alpha,\beta;(Y_{t_k})) $ which is, using (\ref{def:Nk}) and (\ref{eq:Zkbis}),
 \begin{equation} \label{eq:likY}
 l(\alpha,\beta;(Y_{t_k}))= -\frac{1}{2} \som{k=1}{n}\log(det \ S_{k}^
{\alpha,\beta}) -\frac{1}{2 \epsilon^2\Delta}\sum_{k=1}^{n} 
\trans{N_{k}(Y,\alpha)}(S_{k}^{\alpha,\beta})^{-1} N_{k}(Y,\alpha).
\end{equation}

\section{Parametric inference for fixed sampling interval}

For the diffusion parameter $\beta$, all existing results for discretized observations on a fixed sampling interval are provided in the context of the asymptotics $\Delta \rightarrow 0$ ($T=n\Delta$). In this section we focus on a different asymptotics ($\epsilon \rightarrow 0$) as $\Delta$ is assumed to be fixed. We build a contrast process based on the the functions $N_k(X, \alpha)$ defined 
in (2.13). Except for some specific cases (e.g. linear drift in the diffusion process), the two 
deterministic quantities $ x_{\alpha}(.), \Phi_{\alpha}(.,.)$ appearing 
in the  $N_k$'s are not explicit and are approximated by solving numerically an ODE with dimension
$p \times(p+1)$.

\subsection{ One-dimensional Ornstein-Uhlenbeck process}
\label{OU}
The one dimensional Ornstein-Uhlenbeck process is an appropriate illustration of the limitations imposed by the assumption $\Delta$ fixed.
Indeed, assuming that $\alpha$ is known and equal to $\alpha_0$, the diffusion process $(X_t)_{t\in[0,T]}$ following  $dX_t=\alpha_0X_tdt+\epsilon\beta dB_t,X_0=x_0\in \R$ is equal to its Gaussian approximation ($X_t=\xzero{}+\epsilon g_{\alpha_0,\beta}(t)$), and $l(\alpha_0,\beta)$ is then the log-likelihood of Gaussian observations. 
Noting that $S_k^{\alpha_0,\beta}=\beta^2\frac{\left(e^{2\alpha_0\Delta}-1\right)}{2\alpha_0 \Delta}$, $\xzero{}=x_0e^{\alpha_0 t}$ and $\fyzero{k}{k-1}=e^{\alpha_0\Delta}$, the maximum likelihood estimator of $\beta$ is given by
\[\hat{\beta}^2_{\epsilon,\Delta}=\frac{2\alpha_0}{\epsilon^2(e^{2\alpha_0\Delta}-1)}\som{k=1}{n}\left(X_{t_{k}}-e^{\alpha_0 \Delta}X_{t_{k-1}}\right)^2.\]
Under $\Pzero$, $\hat{\beta}^2_{\epsilon,\Delta}=\beta_0^2\som{k=1}{n}U_k^2$, 
where $U_k^2=\frac{2\alpha_0}{e^{2\alpha_0\Delta}-1}\left(\integ{t_{k-1}}{t_k}e^{\alpha_0(t_k-s)}dB_s\right)^2$.\\ Hence, $(U_k)_{1\leq k\leq n}$ are i.i.d. random variables $\mathcal{N}(0,1)$, and $\hat{\beta}_{\epsilon,\Delta}^2$ is  unbiased for all $\epsilon$ but has no other properties as $\epsilon\rightarrow0$.


\subsection{General case ($\beta$ unknown)}

In the case where we have no information on $\beta$, it is quite natural to consider a contrast process derived from the conditional least squares for $(Y_{t_k})$, which does not depend on $\beta$ and is defined using (\ref{eqphi}) and (\ref{def:Nk}) by
\begin{equation}\label{def:cls}
	\begin{array}{rcl}
\bar{U}_{\epsilon,\Delta}\left(\alpha;(X_{t_k})\right)
=\bar{U}_{\epsilon,\Delta}(\alpha) 
&=&\frac{1}{\Delta}\som{k=1}{n}\trans{N_{k}(X,\alpha)}
N_{k}(X,\alpha).
\end{array}
\end{equation} 
Then, the conditional least square estimator is defined as any solution of, 
\begin{equation}\label{est:cls}
 \bar{\alpha}_{\epsilon,\Delta}=\underset{\alpha\in K_a }{argmin} \
\bar{U}_{\epsilon,\Delta}\left(\alpha,(X_{t_k})\right).
\end{equation}
Let us also define,
\[\bar{K}_{\Delta}(\alpha_0,\alpha) = \frac{1}{\Delta}\som{k=1}{n}\trans{N_k(x_{\alpha_0},\alpha)}N_k (x_{\alpha_0},\alpha).\]
Clearly, $\bar{K}_\Delta(\alpha_0,\alpha)\geq 0$ and $\bar{K}_\Delta(\alpha_0,\alpha_0)=0$. Now, $\bar{K}_\Delta(\alpha_0,\alpha)=0 $ if for all $ k$,
 $x_\alpha(t_k)-x_{\alpha_0}(t_k)= \Phi_\alpha(t_k,t_{k-1})(x_\alpha(t_{k-1})-x_{\alpha_0}(t_{k-1}))$.
The matrix $\Phi_\alpha(t_k,t_{k-1})$ being invertible, this is the idenfiability assumption (S4').\\
\begin{lemma}
 Assume {\bf (S1)}, {\bf (S2)}. Then, under $\mathbb{P}_{\theta_0}$, 
\begin{equation}\label{barK} 
\bar{U}_{\epsilon,\Delta}(\alpha) \tend{\epsilon}{0}
\bar{K}_{\Delta}(\alpha_0,\alpha)
\mbox{ in probability.}
\end{equation}
\end{lemma}
Using (\ref{def:Nk}) and the stochastic Taylor formula (\ref{taylor_stochastique}) the proof is immediate.\\
In order to study $\bar{\alpha}_{\epsilon,\Delta}$, we define for $1\leq i\leq a$ and for $1 \leq k \leq n$, 
\begin{equation}\label{def:D_k}
D_{k,i}(\alpha)=\frac{1}{\Delta}\left[-\deriv{x_\alpha(t_{k})}{\alpha_i}(\alpha)+\Phi_{\alpha}(t_{k},t_{k-1})\deriv{
x_\alpha(t_{k-1})}{\alpha_i}(\alpha)\right]\in \R^p,\end{equation}
and 
\begin{equation}\label{def:Mdelta}
M_\Delta(\alpha)=\left(\Delta \som{k=1}{n} \trans{D}_{k,i}(\alpha)D_{k,j}(\alpha)\right)_{1\leq i,j\leq a}\in M_a(\R).
\end{equation}
\begin{prop}\label{prop:cls}
 Assume {\bf (S1)-(S3)} and {\bf (S4')}. Then, under $\mathbb{P}_{\theta_0}$,\\
(i) $\bar{\alpha}_{\epsilon,\Delta}\tend{\epsilon}{0}\alpha_0$ in probability.\\
(ii) If $M_\Delta(\alpha_0)$ is invertible, $\epsilon^{-1}\left(
\bar{\alpha}_{\epsilon,\Delta}-\alpha_0\right)\tend{\epsilon}{0}
\mathcal{N}(0,J^{-1}_\Delta(\alpha_0,\beta_0))$ in distribution, with\\
\begin{equation}\label{def:Jdelta}J_\Delta(\alpha_0,\beta_0)=M_\Delta(\alpha_0)\left(\Delta\som{k=1}{n}\trans{D_{k,i}
(\alpha_0)}S^{\theta_0}_{k}D_{k,j}(\alpha_0)\right)^{-1}_{1\leq i,j\leq a}\trans{M}_\Delta(\alpha_0).\end{equation}
\end{prop}
\begin{proof}
The proof of (i) is classical and relies on the control of the continuity modulus of $\bar{U}_{\Delta,\epsilon}(\alpha,X_{t_{k}}(\omega))$ (see Appendix \ref{proof:cls} for details).\\
Let us just study $\frac{1}{\epsilon}\deriv{\bar{U}_{\epsilon}(\alpha_0)}{\alpha}$. Expanding
$\deriv{\bar{U}_{\epsilon}(\bar{\alpha}_{\epsilon})}{\alpha}$ in Taylor series at point $\alpha_0$ yields\\
$0=\frac{1}{\epsilon}\deriv{\bar{U}_{\epsilon}(\alpha_0)}{\alpha}+\left[
\deriv{^2\bar{U}_{\epsilon}}{\alpha^2}(\alpha_0) + \integ{0}{1}\left(\deriv{^2\bar{U}_{\epsilon}}{\alpha^2}(\alpha_0+
t(\bar{\alpha_{\epsilon,\Delta}}-\alpha_0))-\deriv{^2\bar{U}_{\epsilon}}{\alpha^2}(\alpha_0)\right)dt\right]\epsilon^{-1
}{(\bar{\alpha}_{\epsilon}-\alpha_0)}.$\\ 
First, we study the term \\$\frac{1}{\epsilon}\deriv{\bar{U}_{\epsilon}(\alpha_0)}{\alpha}=2\sqrt{\Delta}\som{k=1}{n}\left(\frac{1}{\Delta}\trans{
\deriv{N_{k}(X,\alpha)}{\alpha}(\alpha_0)}\right)\left(\frac{1}{\epsilon\sqrt{\Delta}}N_{k}(X,\alpha_0)\right)$\\
Using (\ref{taylor_stochastique}) and (\ref{eq:Zk}),
\begin{equation}\label{limN_k}
\frac{1}{\epsilon\sqrt{\Delta}}N_k(X,\alpha_0)\tend{\epsilon}{0}Z_k^{\theta_0} \mbox{ in
$\P_{\theta_0}$-probability.}\end{equation}
Now, by (\ref{def:Nk}) and (\ref{def:D_k}), 
$\deriv{N_{k}(X,\alpha)}{\alpha_i}(\alpha_0)=\Delta
D_{k,i}(\alpha_0)+\deriv{\Phi_\alpha(t_{k},t_{k-1})}{\alpha_i}(\alpha_0)\left[X_{t_{k-1}}-x_{\alpha_0}(t_{k-1})\right]$, and based on (\ref{tsto1}) we obtain
 \begin{equation}\label{limdNk}
 \deriv{N_k(X,\alpha)}{\alpha_i}(\alpha_0)\tend{\epsilon}{0}\Delta D_{k,i}(\alpha_0).\end{equation}
By Slutsky's Lemma,
$\left(\frac{1}{\Delta}\trans{\deriv{N_{k}(X,\alpha)}{\alpha_i}(\alpha_0)}\right)\left(\frac{1}{\epsilon\sqrt{\Delta}}N_{k
}(X,\alpha_0)\right)\tend{\epsilon}{0}D_{k,i}(\alpha_0)Z_k^{\theta_0}$, and definition (\ref{S_k+1}) yields that, under $\P_{\theta_0}$, as $\epsilon \rightarrow 0$, 
\begin{equation*}\label{limdU} 
\frac{1}{\epsilon}\deriv{\bar{U}_{\epsilon}}{\alpha}(\alpha_0)\tend{\epsilon}{0}
\mathcal{N}\left(0,4\Delta\left(\som{k=1}{n}\trans{D_{k,i}(\alpha_0)}S^{\alpha_0,\beta_0}_{k}D_{k,j}(\alpha_0)\right)_{1\leq i ,j\leq a}\right)\mbox{ in distribution.}
\end{equation*}
In Appendix \ref{proof:cls} we prove, by using the matrix defined in (\ref{def:Mdelta}), that
$\P_{\theta_0}$- a.s.,
\begin{equation*}\label{lim d2U} 
 \deriv{^2\bar{U}_{\epsilon}}{\alpha_i\partial \alpha_j}(\alpha_0)\tend{\epsilon}{0}2 M_\Delta(\alpha_0)_{i,j} \; \mbox{ and}
\end{equation*}
$\underset{t\in[0,1]}{sup}
\norm{\deriv{^2\bar{U}_{\epsilon}}{\alpha^2}(\alpha_0+t(\bar{\alpha}_{\epsilon,\Delta}-\alpha_0))-\deriv{^2\bar{U}_{
\epsilon}}{\alpha^2}(\alpha_0)}\tend{\epsilon}{0}0,$ which completes the proof of (ii).
\end{proof}
It is well known that the Fisher Information matrix for a continuously observed diffusion in the asymptotics of $\epsilon\rightarrow 0$ is (see e.g. \cite{kut84})
\begin{equation}\label{Fisher}I_b(\alpha_0,\beta_0)=\left(\integ{0}{T}\trans{\deriv{b}{\alpha_i}(\alpha_0,x_{\alpha_0}(s))}\Sigma^{-1}(\beta_0,x_{\alpha_0}(s))\deriv{b}{\alpha_j}(\alpha_0,x_{\alpha_0}(s)) ds \right)_{1\leq i,j\leq a}.\end{equation}
Setting $F_b(\alpha_0,M)=\left(\integ{0}{T}\trans{\deriv{b}{\alpha_i}(\alpha_0,x_{\alpha_0}(s))}M(s)\deriv{b}{\alpha_j}(\alpha_0,x_{\alpha_0}(s)) ds \right)_{1\leq i,j\leq a}$, we have that $M_\Delta(\alpha_0)\rightarrow F_b(\alpha_0,I_p)$ and
$J_\Delta(\alpha_0,\beta_0)\rightarrow F_b(\alpha_0,I_p)(F_b(\alpha_0,\Sigma(\beta_0,x_{\alpha_0}(\cdot))))^{-1}\trans{F}_b(\alpha_0,I_p)$ as $\Delta\rightarrow 0$. This is different from the Fisher Information matrix \\$I_b(\alpha_0,\beta_0)=F_b(\alpha_0,\Sigma^{-1}(\beta_0,x_{\alpha_0}(\cdot)))$, but possesses the right rate of convergence. 

\subsection{Case of additionnal information on $\beta$}
In this section we will consider successively the case where $\beta$ is a known regular function of $\alpha$ and the multiplicative case for parameter $\beta$ which applies to Ornstein-Uhlenbeck or Cox-Ingersoll-Ross models for examples. In the former context, one particular subcase, interesting in applications, is given by $\alpha=\beta$ (see Section \ref{epidemics})).

\subsubsection{Case of $\beta=f(\alpha)$, f known}

In many applicative situations, such as the modelling of epidemic spread (see Section \ref{epidemics}), we have $\beta=\alpha$. Using a contrast depending on $\beta$ through $\alpha$ leads to the optimal asymptotic Information. We regroup these cases in a more general formulation with $\beta=f(\alpha)$, where f is known and regular.

Since the Gaussian  process $(Y_t)$ is a good approximation of $(X_t)$ for small $\epsilon$ (Theorem \ref{theo:Azen}),
we use the likelihood (\ref{eq:likY}) to  derive a contrast process for $(X_{t_k})$. The sampling interval $\Delta$ being fixed, the first  term of (\ref{eq:likY}) converges to a finite limit as $\epsilon \rightarrow 0$. This leads to the contrast process, using (\ref{S_k+1})
\begin{equation}\label{utild}
	\begin{array}{rcl}
\tilde{U}_{\Delta,\epsilon}\left(\alpha;(X_{t_k})
\right)&=&\frac{1}{\Delta}\som{k=1}{n}\trans{N_{k}(X,\alpha)}
(S_{k}^{\alpha,f(\alpha)})^{-1}N_{k}(X,\alpha).
\end{array}
\end{equation} 
Then, we can define 
\begin{equation}\label{alphatild}
\tilde{\alpha}_{\epsilon,\Delta}=\underset{\alpha\in K_a }{argmin} \ \tilde{U}_{\epsilon,\Delta}\left(\alpha,(X_{t_k})\right).
 \end{equation}
Clearly, under $(S1)-(S3)$, $\tilde{U}_{\Delta,\epsilon}(\alpha)\tend{\epsilon}{0}\tilde{K}_{\Delta}(\alpha_0,\alpha)$, where
\begin{equation}
\tilde{K}_{\Delta}(\alpha_0,\alpha)=\frac{1}{\Delta}\som{k=1}{n}\trans{N_{k}(x_{\alpha_0},\alpha)}(S_{k}^{\alpha,f(\alpha)})^{-1}N_{k}(x_{\alpha_0},\alpha).
\end{equation}
Assumption {\bf(S4')} ensures that $\tilde{K}_\Delta(\alpha_0,\alpha)$ is non negative and has a strict minimum at $\alpha=\alpha_0$.
\begin{prop}\label{prop:contrast_betaknown}
 Assume {\bf(S1)-(S3),(S4')}. Then,\\
(i) $\tilde{\alpha}_{\epsilon,\Delta}\tend{\epsilon}{0}\alpha_0$ in $\mathbb{P}_{\theta_0}$-probability.\\
(ii) If $I_\Delta(\alpha_0,\beta_0)$ is invertible, $\epsilon^{-1}\left( \tilde{\alpha}_{\epsilon,\Delta}-\alpha_0\right)\tend{\epsilon}{0}\mathcal{N}(0,I^{-1}_\Delta(\alpha_0,\beta_0))$, under $\mathbb{P}_{\theta_0}$ in distribution, with\\
\begin{equation}\label{def:Idelta}I_\Delta(\alpha_0,\beta_0)=\Delta\left(\som{k=1}{n}\trans{D_{k,i}(\alpha_0)}\left(S^{\theta_0}_{k}\right)^{-1}D_{k,j}(\alpha_0)\right)_{1\leq i,j \leq a}\end{equation}
\end{prop}
The proof of (i) is a repetition of the proof of Proposition (\ref{prop:cls}).
The proof of (ii) relies again on the two properties. Under $\mathbb{P}_{\theta_0}$,
${\epsilon}^{-1}\deriv{U_{\Delta, \epsilon}}{\alpha}(\alpha_0)  \tend{\epsilon}{0}\mathcal{N}\left(0,4I_\Delta(\alpha_0,\beta_0)\right)$ in distribution and
$\deriv{^2U_{\Delta,\epsilon}}{\alpha^2}(\alpha_0)\tend{\epsilon}{0}2I_\Delta(\alpha_0,\beta_0)$ in probability.
Contrary to Proposition \ref{prop:cls}, additional terms appear due to the derivation of $S_k^{\alpha,f(\alpha)}$. Those terms are controlled using $\alpha\rightarrow \Phi_\alpha(t_{k},t)$ and $\alpha\rightarrow \Sigma(f(\alpha),x_\alpha(t))$ regularities. Details of the proof are given in Appendix \ref{proof:contrast_betaknown}.\\
\begin{rem}
 Contrary to the previous contrast (\ref{def:cls}), the Covariance matrix is asymptotically optimal in the sense that 
$I_\Delta(\alpha_0,\beta_0)\tend{\Delta}{0}I_b(\alpha_0,\beta_0)$ where $I_b$ is define in (\ref{Fisher}).
\end{rem}
\subsubsection{The multiplicative case ($\Sigma(\beta,x)= f(\beta) \Sigma_0(x)$)}
\label{multiplicative}
The case of $\Sigma(\beta,x)= f(\beta) \Sigma_0(x)$ with $f(\cdot)$ a stricly positive known function of  $C(\R^b,\R_+^*)$ often occurs in practice. Noting that  $S_k^{\alpha,\beta}= f(\beta)S_k^{\alpha,0}$ with \\$S_k^{\alpha,0}=\frac{1}{\Delta}\integ{t_{k-1}}{t_{k}}\Phi_\alpha(t_{k},s)\Sigma_0(x_\alpha(s))
\trans{\Phi_\alpha(t_{k},s)}ds$. Define the contrast process\\ $\tilde{U}_{\Delta,\epsilon}\left(\alpha;(X_{t_k})
\right)=\frac{1}{\Delta}\som{k=1}{n}\trans{N_{k}(X,\alpha)}(S_{k}^{\alpha,0})^{-1}N_{k}(X,\alpha)$, then
\begin{corol}
 Assume $(S1)-(S3)-(S4')$, Then, under $\Pzero$, as $\epsilon\rightarrow0$,\\
(i) $\tilde{\alpha}_{\epsilon,\Delta}\tend{\epsilon}{0}\alpha_0$ in $\mathbb{P}_{\theta_0}$-probability.\\
(ii) If $I_\Delta(\alpha_0,\beta_0)$ is invertible, $\epsilon^{-1}\left( \tilde{\alpha}_{\epsilon,\Delta}-\alpha_0\right)\tend{\epsilon}{0} \mathcal{N}(0,I^{-1}_\Delta(\alpha_0,\beta_0))$ in distribution.
\end{corol}
Its study is similar to $\beta=f(\alpha)$, with a substitution of $S_{k}^{\alpha,f(\alpha)}$ by $S_{k}^{\alpha,0}$ in (\ref{utild}).
\section{Parametric inference for small sampling interval}
We assume now that $\Delta=\Delta_n\rightarrow 0$, so that the number of observations $n=T/\Delta_n$ goes to infinity. The results obtained by Gloter and S\o rensen (\cite{glo09}) state that, under the  additional condition ($\exists \rho >0, \ \Delta_n^\rho / \epsilon$ bounded), the rates of convergence for $\alpha,\beta$ are respectively $\epsilon^{-1}$ and $\frac{1}{\sqrt{\Delta_n}}$. Indeed, considering the one dimensional Ornstein-Uhlenbeck process the estimator $\hat{\beta}^2_{\epsilon,\Delta}$ obtained in Section 3.1 is still the MLE, is consistent and satisfies $\sqrt{n}\left(\hat{\beta}^2_{\epsilon,\Delta}-\beta_0^2\right)\tend{\epsilon,\Delta}{0}\mathcal{N}(0,2\beta_0^4)$.\\

In the sequel, we follow \cite{sor03} and \cite{glo09}, which allow to study  contrast estimators of parameters which converge at different rates:
we prove the consistency of $ \check{\alpha}_{\epsilon,\Delta}$ in Proposition \ref{CV-contraste} and the tightness of the sequence $(\check{\alpha}_{\epsilon,\Delta}-\alpha_0)/\epsilon$ in Proposition \ref{borneproba}. From this, we deduce  the consistency of $ \check{\beta}_{\epsilon,\Delta}$ in Proposition \ref{CV-beta}. Asymptotic normality  for both estimators is finally proved in Theorem \ref{theo:normalite}.\\
For clarity, we omit the index $n$ in $\Delta_n$.\\
Using (\ref{def:Nk}), let us consider now  the contrast process 
$\check{U}_{\epsilon,\Delta}\left((\alpha,\beta),(X_{t_{k}})\right)= \check{U}_{\epsilon,\Delta} (\alpha,\beta)$
\begin{equation}
	\label{def:Ucheck}
\check{U}_{\epsilon,\Delta} (\alpha,\beta)=\som{k=1}{n}\log \;det \;\Sigma(\beta,X_{t_{k-1}})+
\frac{1}{\epsilon^2\Delta}\som{k=1}{n}\trans{N_k(X,\alpha)}\Sigma^{-1}(\beta,X_{t_{k-1}})N_k(X,\alpha).
		\end{equation}
The minimum contrast estimators are defined as any solution of  
\begin{equation}\label{def:estimateurs_chech}
(\check{\alpha}_{\epsilon,\Delta},\check{\beta}_{\epsilon,\Delta})=\underset{(\alpha,\beta)\in \Theta }{argmin} \ \check{U}_{\epsilon,\Delta}(\alpha,\beta).
\end{equation}
For studying these estimators, we need to state some lemmas on the behaviour of $\Nk$.
\subsection{Asymptotic properties of $\Nk$}
Clearly, as $\epsilon$ goes to zero, $\Nk$ converges to $N_k(x_{\alpha_0},\alpha)$ by \eqref{tsto1} under $\Pzero$.
Let us define the function
\begin{equation} \label{def:Gamma}
      	\Gamma(\alpha_0,\alpha;t)=b(\alpha_0,x_{\alpha_0}(t))-b(\alpha,x_\alpha(t))-\deriv{b}{x}(\alpha,x_{\alpha}(t))(x_{\alpha_0}(t)-x_\alpha(t))\; \in \R^p.
      \end{equation}
Then, functions $\left(N_k(x_{\alpha_0},\alpha)\right)_{k\leq n}$ satisfy
\begin{lemma}\label{prop:Nk-Gamma}
Under {\bf(S2)}
	\begin{equation*}
 \super{k\in\{1,..,n\},\alpha\in K_a} \norm{\frac{N_k(x_{\alpha_0},\alpha)}{\Delta}- \Gamma(\alpha_0,\alpha;t_{k-1})}\tend{\Delta}{0}0.
\end{equation*}
\end{lemma}
\begin{proof}
First, note that $\Nkd$ defined in \eqref{def:Nk} writes
\begin{equation}\label{prop:decompoNk}
	\Nkd=\left(\xz{k}-\xz{k-1}\right)-\left(\x{k}-\x{k-1}\right)-\left[\fy{k}{k-1}-I_p\right]\left(\xz{k-1}-\x{k-1}\right).
\end{equation}
Hence, using \eqref{def:Gamma} we have \\
$\begin{array}{lll}\frac{1}{\Delta}N_k(x_{\alpha_0},\alpha)&=&\Gamma(\alpha_0,\alpha;t_{k-1})+\frac{1}{\Delta}\left(x_{\alpha_0}(t_k)-x_{\alpha_0}(t_{k-1})\right)-b(\alpha_0,x_{\alpha_0}(t_{k-1}))\\
  & &-\frac{1}{\Delta}\left(x_{\alpha}(t_k)-x_{\alpha}(t_{k-1})\right)+b(\alpha,x_{\alpha}(t_{k-1}))\\
& &+\left[\frac{1}{\Delta}\left(\fy{k}{k-1}-I_p\right)-\deriv{b}{x}(\alpha,x_\alpha(t_{k-1}))\right]\left(x_{\alpha_0}(t_{k-1})-\x{k-1}\right).
\end{array}$\\
The uniform approximation is then obtained using the analytical properties \eqref{bound:phi-id}, \eqref{bound:x} of $x_\alpha$ and $\Phi_\alpha$ given in Appendix \ref{bound:analytics}.
\end{proof}
Let us now study the properties of $\Nk$.
\begin{lemma}\label{prop:Nk_deltapetit}
Assume {\bf(S1)-(S3)}. Then, under $\P_{\theta_0}$, for all k ($1\leq k\leq n$),
	\begin{equation*}
 \frac{1}{\Delta}\left[N_k(X,\alpha)-N_k(X,\alpha_0)\right]=\frac{1}{\Delta}N_k(x_{\alpha_0},\alpha)+\epsilon\norm{\alpha-\alpha_0}\eta_k,
\end{equation*}
where $\eta_k=\eta_k(\alpha_0,\alpha,\epsilon,\Delta)$ is $\Ft{k-1}$-measurable and satisfies that, under $\Pz$, as $\epsilon,\Delta\rightarrow0$, $\super{k\in\{1,..,n\}, \alpha\in K_a}\norm{\eta_k}$ is bounded in probability.
\end{lemma}
\begin{proof}
	Using \eqref{taylor_stochastique} and \eqref{def:Nk}, $\Nk$ writes\\
$\Nk=\Nkzero +N_k(x_{\alpha_0},\alpha)+\left(\fyzero{k}{k-1}-\fy{k}{k-1}\right)\epsilon R^{1,\epsilon}_{\theta_0}(t_{k-1})$.\\
 Applying \eqref{bound:phi-id} yields that \\$\frac{1}{\Delta}\norm{\fyz{k}{k-1}-\fy{k}{k-1}}\leq 2\norm{\dbxz{k-1}-\dbx{k-1}}\leq K\norm{\alpha-\alpha_0}$. Assumption {\bf(S3)} ensures that $(t,\alpha)\rightarrow \dbx{}$ is uniformly continuous on $[0,T]\times K_a$, and \eqref{taylor_stochastique} that $\super{t\in [0,T]}\norm{R^{1,\epsilon}_{\theta_0}(t)}$ is bounded in probability under $\Pz$. The proof is achieved setting $\eta_k=\frac{\fyzero{k}{k-1}-\fy{k}{k-1}}{\Delta\norm{\alpha-\alpha_0}}R^{1,\epsilon}_{\theta_0}(t_{k-1})$ and noting that $R^{1,\epsilon}_{\theta_0}(t_{k-1})$ is $\Ft{k-1}$-measurable.
\end{proof}
The following Lemma concerns the properties of $\Nkz$
\begin{lemma}\label{prop:decompositionGS}
 Assume {\bf(S1)-(S3)}. Then, under $\P_{\theta_0}$,
\begin{equation*}
\Nkzero=\epsilon\sigma(\beta_0,X_{t_{k-1}})\left(B_{t_{k}}-B_{t_{k-1}}\right)+E_k,\end{equation*} where $E_k=E_k(\alpha_0,\beta_0)$ satisfies, for $m\geq2$, 
$\Ec{\norm{E_k}^m}{k-1}\leq C\epsilon^m \Delta^m$.
\end{lemma}
The proof of Lemma \ref{prop:decompositionGS} follows the proof of \cite{glo09} and is given in Appendix \ref{prooflemme:decompositionGS}.\\
 The properties of the derivatives of $N_k(X,\alpha)$ are given in the following Lemma.
\begin{lemma}\label{prop:derivatives}
 Assume {\bf(S1)-(S3)}. Then, for all $i,j$, $1\leq i,j\leq a$, as $\epsilon,\Delta \rightarrow0$\\
(i) $\frac{1}{\Delta}\deriv{\Nk}{\alpha_i}(\alpha_0)=-\deriv{b}{\alpha_i}(\alpha_0,\xz{k-1})+\epsilon \, \zeta_{k,i} +r_{k,i}$
where \\$\zeta_{k,i}=\zeta_{k,i}(\alpha_0,\epsilon,\Delta)$ is $\Ft{k-1}$-measurable and satisfies that $\super{k\in\{1,..,n\}}\norm{\zeta_{k,i}}$ is bounded in $\Pz$-probability as $\epsilon,\Delta\rightarrow0$, and $r_{k,i}=r_{k,i}(\alpha_0,\Delta)$ is deterministic and satisfies $\super{k\in\{1,..,n\}}\norm{r_{k,i}}\tend{\Delta}{0}0$.\\
(ii) $\frac{1}{\Delta}\deriv{^2\Nk}{\alpha_i\partial \alpha_j}(\alpha_0)$ is bounded in $\Pz$-probability.
\end{lemma}
\begin{proof}
 Let us first prove (i). Using \eqref{tsto1} and \eqref{prop:decompoNk}, we get\\
$\begin{array}{ll}\deriv{\Nk}{\alpha_i}(\alpha_0)=&-\deriv{\x{k}}{\alpha_i}(\alpha_0)+\deriv{\x{k-1}}{\alpha_i}(\alpha_0)-\left[\fyz{k}{k-1}-I_p\right]\deriv{\x{k-1}}{\alpha_i}(\alpha_0)\\
  &+\epsilon\deriv{\fy{k}{k-1}}{\alpha_i}(\alpha_0) R^{1,\epsilon}_{\theta_0}(t_{k-1}).
 \end{array}$\\
Set $\zeta_{k,i}=\frac{1}{\Delta}\deriv{\fy{k}{k-1}}{\alpha_i}(\alpha_0) R^{1,\epsilon}_{\theta_0}(t_{k-1})$. Using \eqref{taylor_stochastique} and \eqref{bound:dphi/dalpha} we obtain, as $\Delta \rightarrow 0$, that $\norm{\frac{1}{\Delta}\deriv{\fy{k}{k-1}}{\alpha_i}(\alpha_0)}\leq 2 \norm{\deriv{^2b(\alpha,x_\alpha(t_{k-1}))}{\alpha_i\partial x}(\alpha_0,\xz{k-1})}$,$\super{t\in[0,T]}\norm{R^{1,\epsilon}_{\theta_0}(t)}$ is bounded in $\Pz$-probability. It remains to study the deterministic part\\ $E_{k,i}=-\deriv{\x{k}}{\alpha_i}(\alpha_0)+\deriv{\x{k-1}}{\alpha_i}(\alpha_0)-\left[\fyz{k}{k-1}-I_p\right]\deriv{\x{k-1}}{\alpha_i}(\alpha_0).$\\
 According to \eqref{bound:phi-id} and \eqref{bound:dx/dalpha}, as $\Delta\rightarrow0$,\\
 $\frac{1}{\Delta}\left(\fyz{k}{k-1}-I_p\right)$ 
(resp. $\frac{1}{\Delta}\left(\deriv{\x{k}}{\alpha_i}-\deriv{\x{k-1}}{\alpha_i}\right)(\alpha_0)$ ) is approximated by 
$\dbxz{k-1}$ (resp. $-\deriv{b(\alpha,x_\alpha(t_{k-1}))}{\alpha_i}(\alpha_0)$ ). Noting that \\$\deriv{b(\alpha,\x{})}{\alpha_i}(\alpha_0)=\deriv{b}{\alpha_i}(\alpha_0,\xzero{})+\deriv{b}{x}(\alpha_0,\xzero{})\deriv{\x{}}{\alpha_i}(\alpha_0)$, we get that\\ $E_{k,i}=-\deriv{b(\alpha,x_\alpha(t_{k-1}))}{\alpha_i}(\alpha_0)+r_{k,i}$ with $\super{k\in\{1,..,n\}}\norm{r_{k,i}}\tend{\Delta}{0}0$, which achieves the proof.
The proof of (ii) is given in Appendix \ref{proofii}.
\end{proof}
\subsection{Study of the contrast process $\check{U}_{\epsilon,\Delta}$}
First, consider the estimation of parameters present in the drift coefficient. Using \eqref{def:Gamma}, we define
\begin{equation}\label{def:K1}
K_1(\alpha_0,\alpha;\beta)=\integ{0}{T}\trans{\Gamma(\alpha_0,\alpha;t)}\Sigma^{-1}(\beta,x_{\alpha_0}(t))\Gamma(\alpha_0,\alpha;t)dt.
\end{equation}
Note that $K_1$ is non negative and by {\bf (S4)}, if $\alpha \neq \alpha_0 $, the function  $\Gamma(\alpha_0,\alpha,.)$ is non identically null. Thus, $K_1(\alpha_0,\alpha,\beta)$ is equal to 0 if and only if  $\alpha=\alpha_0$, and defines a contrast function for all $\beta$.
\begin{prop}\label{CV-contraste}
Assume {\bf(S1)-(S4)}. Then, as $\epsilon\rightarrow 0$ and $\Delta\rightarrow 0 $, under $\Pz$, using definition \eqref{def:Ucheck} for $\check{U}_{\epsilon,\Delta}$ \\
(i) $ \super{\theta\in \Theta}\abs{\epsilon^2\left(\check{U}_{\epsilon,\Delta}(\alpha,\beta)-\check{U}_{\epsilon,\Delta}(\alpha_0,\beta)\right)-K_1(\alpha_0,\alpha;\beta)} \rightarrow 0$ in probability; \\
(ii) $\check{\alpha}_{\epsilon,\Delta}\tend{\epsilon,\Delta}{0}\alpha_0$ in probability.
\end{prop}
\begin{proof}
Let us prove (i). Using \eqref{def:Ucheck},we get\\ {\small $\epsilon^2\left(\check{U}_{\epsilon,\Delta}(\alpha,\beta)-\check{U}_{\epsilon,\Delta}(\alpha_0,\beta)\right)=\frac{1}{\Delta}\som{k=1}{n}\trans{\left[\Nk -\Nkzero\right]}\Sig\left[\Nk +\Nkzero\right],$}\\ Using that $N_k(x_{\alpha_0},\alpha_0)=0$, an application of Lemma \ref{prop:Nk_deltapetit} yields\\
{\small$\epsilon^2\left(\check{U}_{\epsilon,\Delta}(\alpha,\beta)-\check{U}_{\epsilon,\Delta}(\alpha_0,\beta)\right)=\Delta\som{k=1}{n}\trans{ \frac{N_k(x_{\alpha_0},\alpha)}{\Delta}}\Sigdet{k-1} \frac{N_k(x_{\alpha_0},\alpha)}{\Delta}+R(\alpha_0,\alpha,\beta;\epsilon,\Delta)$.}
The first term of the above formula is a Riemann sum which converges by Lemma \ref{prop:Nk-Gamma} to the function $K_1(\alpha_0,\alpha,\beta)$ defined in \eqref{def:K1} as $\Delta \rightarrow 0$. This convergence is uniform with respect to the parameters. Let us now study the remainder term. Using Lemma \ref{prop:Nk_deltapetit}, we get that\\
$R(\alpha_0,\alpha,\beta;\epsilon,\Delta)=T_1+T_2+T_3$, where\\
$T_1=\Delta\som{k=1}{n}\trans{\frac{N_k(x_{\alpha_0},\alpha)}{\Delta}}\left[\Sig-\Sigdet{k-1}\right]\frac{N_k(x_{\alpha_0},\alpha)}{\Delta}$,\\
$T_2=\Delta\epsilon \norm{\alpha-\alpha_0}\som{k=1}{n}\trans{V_{k}}\eta_k$, with $V_{k}=\Sig\left(\frac{N_k(x_{\alpha_0},\alpha)}{\Delta}+\epsilon \norm{\alpha-\alpha_0}\eta_k\right)$ ,\\
$T_3=2\som{k=1}{n}\trans{V_{k}}\Nkzero$.\\
Using Lemma \ref{prop:Nk-Gamma} yields \\$\abs{T_1}\leq 2n \Delta \super{ t \in [0,T],\alpha \in K_a}\norm{\Gamma(\alpha_0,\alpha;t)}\super{\beta\in K_b}\norm{\Sig-\Sigdet{k-1}}.$
By Taylor stochastic formula this supremum goes to zero in $\Pz$-probability as $\epsilon \rightarrow 0$. The term $T_2$ contains the random variables $\eta_k$ and $V_k$ which are uniformly bounded in $\Pz$-probability by Lemma \ref{prop:Nk_deltapetit}. Hence $\abs{T_2}\leq \epsilon T \super{k\in \{1,..,n\}, \alpha \in K_a}{\eta_k}\supk{V_k}$ which yields that $T_2$ goes to 0 as $\epsilon,\Delta \rightarrow 0$.
Finally, we prove that $T_3$ goes to zero in $\Pz$-probability, by setting the more general result:
\begin{equation}\label{prop:moments_Nk1}
 \mbox{if $\supk{V_k}<\infty$, }\; \somn \trans{V}_k\Nkz \tend{\epsilon,\Delta}{0}0,\mbox{ in $\Pz$-probability.} 
\end{equation}
Indeed, Lemma \ref{prop:decompositionGS} yields \\
$\norm{\Ec{\trans{V}_k\Nkz}{k-1}}=\norm{\trans{V}_k\Ec{E_k}{k-1}}\leq \supk{V_k}\sqrt{\Ec{\norm{E_k}^2}{k-1}}\leq C\Delta\epsilon$.\\
Using \eqref{prop:moments_Nk2} in Appendix \ref{prooflemme:decompositionGS} yields $\norm{\Ec{(\trans{V}_k\Nkz)^2}{k-1}}\leq C'\Delta \epsilon^2$.\\
Set $X_{n,k}=\trans{V}_k\Nkz$. We get \eqref{prop:moments_Nk1} using an application of Lemma 9 in \cite{gen93} (Lemma \ref{lemma:gc} in Appendix).
All convergences above are uniform with respect to $\theta$ and the proof of (i) is achieved.\\
Let us now prove (ii). The uniformity with respect to $\alpha$ in (i) ensures that the continuity modulus of $\check{U}_{\epsilon,\Delta}$ is dominated, as $\epsilon,\Delta \rightarrow 0$, by the continuity modulus of $K_1$. By compacity of $K_a$, we can extract a sub-sequence of $\check{\alpha}_{\epsilon,\Delta}$, $\left(\check{\alpha}_{\epsilon_k,\Delta_k}\right)_{k\geq1}$ with $\check{\alpha}_{\epsilon_k,\Delta_k}\tend{k}{\infty}\alpha_\infty \in K_a$. Then, by definition (\ref{def:estimateurs_chech}) of $\check{\alpha}_{\epsilon,\Delta}$, $0\leq K_1(\alpha_0,\alpha_\infty,\beta)\leq K_1(\alpha_0,\alpha_0,\beta)$, which yields, by {\bf(S4)}, $\alpha_\infty=\alpha_0$. So any convergent subsequence of $\check{\alpha}_{\epsilon,\Delta}$ goes to $\alpha_0$ which achieves the proof.
\end{proof}
The following Proposition studies the tightness of $\epsilon^{-1}\left( \check{\alpha}_{\epsilon,\Delta}-\alpha_0\right)$ with respect to $\beta$
\begin{prop}\label{borneproba}
 Assume {\bf(S1)-(S4)}. If $I_b(\alpha_0,\beta_0)$ defined in \eqref{Fisher} is invertible,\\ as $\epsilon,\Delta \rightarrow 0$,  $\super{\beta \in K_b}\norm{\epsilon^{-1}\left( \check{\alpha}_{\epsilon,\Delta}-\alpha_0\right)}$ is bounded in $\mathbb{P}_{\theta_0}$-probability.
\end{prop}
Using definition \eqref{Fisher} for $I_b$, the proof given in Appendix \ref{calculation:borneproba} relies on the two properties, for all $\beta \in K_b$
\begin{equation}\label{2} \forall (i,j)\in \{ 1,..,a\}^2,\epsilon^2 \deriv{^2\check{U}_{\epsilon}}{\alpha_i \alpha_j}(\alpha_0,\beta)\tend{\epsilon,\Delta}{0}2I_b(\alpha_0,\beta)_{i,j}\;,\end{equation}
\begin{equation}\label{a} \epsilon \deriv{\check{U}_{\epsilon}(\alpha_0,\beta)}{\alpha}\tend{\epsilon,\Delta}{0}\mathcal{N}\left(0, 4 I_b(\alpha_0,\beta)\right).\end{equation}
For studying the estimation of $\beta$, let us define
\begin{equation}\label{def:K2}
\begin{array}{lll}
K_2(\alpha_0;\beta_0,\beta)&=&\frac{1}{T}\integ{0}{T} Tr(\Sigma^{-1}(\beta,x_{\alpha_0}(t)) \Sigma(\beta_0,x_{\alpha_0}(t))dt \\
&&-\frac{1}{T}\integ{0}{T} \log det (\Sigma^{-1}(\beta,x_{\alpha_0}(t))\Sigma(\beta_0,x_{\alpha_0}(t))) \; dt-p.
\end{array}
\end{equation}
Using  the inequality for invertible symmetric $p\times p$ matrices $A$, $\;Tr(A)-p-\log(det(A))\geq0$,
 with equality if and only if  $A=I_p$,  we get that, for all $\beta$, $ K_2(\alpha_0;\beta_0,\beta)$ is non negative and is equal to 0 if, for all $t$, $\Sigma(\beta_0,x_{\alpha_0}(t))=\Sigma (\beta,x_{\alpha_0}(t))$, which implies $\beta=\beta_0$ by {\bf(S5)}.\\
\begin{prop}\label{CV-beta}
 Assume {\bf(S1)-(S5)}. Then, if $I_b(\alpha_0,\beta_0)$ is invertible, the following holds in $\Pzero$-probability, using \eqref{def:Ucheck}, \eqref{def:estimateurs_chech} and \eqref{def:K2}\\
(i)$\super{\beta\in K_b}\abs{\frac{1}{n}\left(\check{U}_{\Delta,\epsilon}(\check{\alpha}_{\epsilon,\Delta},\beta) -\check{U}_{\Delta,\epsilon}(\check{\alpha}_{\epsilon,\Delta},\beta_0)\right) -K_2(\alpha_0;\beta_0,\beta)}\tend{\epsilon,\Delta}{0}0$\\
(ii) $\check{\beta}_{\epsilon,\Delta}\tend{\epsilon,\Delta}{0}\beta_0.$
\end{prop}
\begin{proof}
Let us first prove (i). Using \eqref{def:Ucheck},we get \\
$\frac{1}{n}\left(\check{U}_{\Delta,\epsilon}(\alpha,\beta) -\check{U}_{\Delta,\epsilon}(\alpha,\beta_0)\right)=A_1(\beta_0,\beta)+A_2(\alpha,\beta_0,\beta)$ with\\
\begin{equation}\label{A1}A_1(\beta_0,\beta)=\frac{1}{n}\som{k=1}{n}log\left(det\left[\Sigma(\beta,X_{t_{k-1}})\Sigma^{-1}(\beta_0,X_{t_{k-1}})\right]\right)\;,\end{equation}\\
\begin{equation}\label{A2}A_2(\alpha,\beta_0,\beta)=\frac{1}{n\Delta\epsilon^2}\som{k=1}{n}\trans{\Nk}\left[\Sigma^{-1}(\beta,X_{t_{k-1}})-\Sigma^{-1}(\beta_0,X_{t_{k-1}})\right]\Nk.\end{equation}\\
Using that, under {\bf(S2)}, $x\rightarrow log\left(det\left[\Sigma(\beta,x)\Sigma^{-1}(\beta_0,x)\right]\right)$ is differentiable on $U$, an application of the Taylor stochastic formula yields\\
$A_1(\beta_0,\beta)=\frac{1}{T}\left( \Delta \som{k=1}{n}log\left(det\left[\Sigma(\beta,x_{\alpha_0}(t_{k-1}))\Sigma^{-1}(\beta_0,x_{\alpha_0}(t_{k-1}))\right]\right)+\epsilon R^{1,\epsilon}_{\alpha_0,\beta,\beta_0}(t_{k-1})\right)$\\
with $\norm{R^{1,\epsilon}_{\alpha_0,\beta,\beta_0}}$ uniformly bounded in $\Pz$ probability. Hence, $A_1(\beta_0,\beta)$, as a Riemann sum, converges to $\frac{1}{T}\integ{0}{T}log\left(det\left[\Sigma(\beta,x_{\alpha_0}(t))\Sigma^{-1}(\beta_0,x_{\alpha_0}(t))\right]\right)dt\;$ as $\epsilon,\Delta \rightarrow 0$.\\
Applying Lemma \ref{prop:decompositionGS} to $\Nkz$ yields\\
$A_2(\alpha_0,\beta_0,\beta)=\frac{\Delta}{T}\somn \trans{U}_kM_kU_k+\frac{1}{\epsilon^2T}\som{k=1}{n}\trans{E_k}\left(\Sigma^{-1}(\beta,X_{t_{k-1}})-\Sigma^{-1}(\beta_0,X_{t_{k-1}})\right)E_k$\\
with $U_k=\frac{1}{\sqrt{\Delta}}\left(B_{t_k}-B_{t_{k-1}}\right)$ and $M_k=\trans{\sigma(\beta_0,X_{t_{k-1}})}\left(\Sigma^{-1}(\beta,X_{t_{k-1}})-\Sigma^{-1}(\beta_0,X_{t_{k-1}})\right)\sigma(\beta_0,X_{t_{k-1}})$.\\
The random vectors $U_k$ are $\mathcal{N}\left(0,I_p\right)$ independant of $\Ft{k-1}$. Hence, using that for $U\sim \mathcal{N}\left(0,I_p\right)$ $E(\trans{U}MU)=Tr(M)$, we get \\
$\Ec{\trans{U}_kM_kU_k}{k-1}=Tr(M_k)=Tr\left(\Sigma^{-1}(\beta,X_{t_{k-1}})\Sigma(\beta_0,X_{t_{k-1}})-I_p\right)$. The first term of $A_2(\alpha_0,\beta_0,\beta)$ converges to $\frac{1}{T}\integ{0}{T}Tr(\Sigma^{-1}(\beta,x_{\alpha_0}(t))\Sigma(\beta_0,x_{\alpha_0}(t))dt -p$. Joining this result with the one for $A_1$ , we obtain consistency towards $K_2$ defined in \eqref{def:K2}.
The detailed proofs for consistency of $A_2(\alpha_0,\beta_0,\beta)$ and control of the error term $A_2(\check{\alpha}_{\epsilon,\Delta},\beta_0,\beta)-A_2(\alpha_0,\beta_0,\beta)$ are given in Appendix \ref{proof:CV-beta}.\\ The proof of (ii) is a repetition of the one of Proposition \ref{CV-contraste}-(ii).
\end{proof}
 Let us now study the asymptotic properties of our estimators. Define the $b \times b$ matrix
\begin{equation}\label{def:Isigma}
 I_\sigma(\alpha_0,\beta_0)_{i,j}=\frac{1}{2T}\integ{0}{T}Tr\left[ \left(\deriv{\Sigma}{\beta_i}\Sigma^{-1}\deriv{\Sigma}{\beta_j}\Sigma^{-1}\right)(\beta_0,x_{\alpha_0}(s)) \right] ds.
\end{equation}
\begin{theo}\label{theo:normalite}
 Assume {\bf(S1)-(S5)}. If $I_b(\alpha_0,\beta_0),I_\sigma(\alpha_0,\beta_0)$ defined in \eqref{Fisher}, \eqref{def:Isigma} are invertible, we have under $\mathbb{P}_{\theta_0}$, in distribution
 \[ \begin{pmatrix}
    \epsilon^{-1}\left(\check{\alpha}_{\epsilon,\Delta}-\alpha_0\right)\\
\sqrt{n}\left(\check{\beta}_{\epsilon,\Delta}-\beta_0\right)
   \end{pmatrix}
\rightarrow \mathcal{N}\left(0,\begin{pmatrix}
                	I_b^{-1}(\alpha_0,\beta_0) & 0 \\
0 & I_\sigma^{-1}(\alpha_0,\beta_0)
                \end{pmatrix}.
\right)\]
\end{theo}
We have already studied the limits as $\epsilon,\Delta \rightarrow 0$ of $\epsilon^2 \deriv{^2\check{U}_{\epsilon}}{\alpha_i \alpha_j}(\alpha_0,\beta_0)$ and  $\epsilon \deriv{\check{U}_{\epsilon}(\alpha_0,\beta_0)}{\alpha}$ in Lemma \ref{borneproba}. These results lead to $\check{\alpha}_{\epsilon,\Delta}$ asymptotic normality. For $\check{\beta}_{\epsilon,\Delta}$ we have to set that $\frac{1}{\sqrt{n}}\deriv{\check{U}_{\Delta,\epsilon}}{\beta_i}(\theta_0)\rightarrow \mathcal{N}\left(0,4 I_\sigma(\theta_0)\right)$ in distribution and $\frac{1}{n} \deriv{\check{U}_{\epsilon,\Delta}}{\beta_i\partial \beta_j}(\theta_0)\rightarrow 2I_\sigma(\theta_0)_{i,j}$ in probability. Finally, for crossed-terms it is sufficient to prove that $\frac{\epsilon}{\sqrt{n}} \deriv{\check{U}_{\epsilon,\Delta}}{\beta_i\partial \alpha_j}(\theta_0)\rightarrow 0$ in probability. Details are provided in Appendix \ref{proof:normalite}

\section{Examples}

\subsection{Exact calculations on Cox-Ingersoll-Ross model (CIR)}
Consider the diffusion on $\R^+$ defined for $\alpha>0$ by
\[dX_t=\alpha X_tdt+\epsilon \beta \sqrt{X_t}dB_t, \; X_0=x_0.\] 
We have $b(\alpha,x)=\alpha x$, $\sigma(\beta,x)=\beta\sqrt{x}$ and $x_\alpha(t)=x_0e^{\alpha t}$. 
The function $\Phi_\alpha$ define in \eqref{eqphi} is explicit with $\Phi_\alpha(t_2,t_1)=e^{\alpha(t_2-t_1)}$. $\Sigma(\beta,x)=\beta^2
x$ and $S_k^{\alpha,\beta}=x_0\beta^2 \; \frac{e^{\alpha\Delta}-1}{\alpha\Delta}e^{\alpha k \Delta}$ depends on $k$ (contrary to the Ornstein-Uhlenbeck process in Section \ref{OU}). This is an AR(1) process, but the noise is not homoscedastic. Let us define $a=e^{\alpha\Delta}$. We have then $\bar{a}_{\epsilon,\Delta}=\left(\som{k=1}{n}X_{t_k}X_{t_{k-1}}\right) / \left(\som{k=1}{n}X_{t_{k-1}}^2\right)$. With notations introduced in previous sections for the different estimators, we have for \eqref{est:cls}:  $\bar{\alpha}_{\epsilon,\Delta}=\frac{1}{\Delta}ln(\bar{a})$. No explicit formula can be obtained for $\tilde{\alpha}_{\epsilon,\Delta}$ and $\check{\alpha}_{\epsilon,\Delta}$ defined in \eqref{alphatild} and \eqref{def:estimateurs_chech}. \\
We can also calculate the asymptotic covariance matrix \eqref{Fisher}: $I_b(\alpha,\beta)=\frac{x_0(e^{\alpha T}-1)}{\beta^2\alpha}$. Noting that $D_k(\alpha)=\Delta e^{\alpha k\Delta}$, we get for \eqref{def:Idelta}: $I_\Delta(\alpha,\beta)=I_b(\alpha,\beta)\times \left(\frac{ln(a)}{a-1}\right)^2a$. Setting $J_b(\alpha,\beta)=\frac{3x_0}{4\alpha\beta^2}\frac{(e^{2\alpha T}-1)^2}{e^{3\alpha T}-1}$, we obtain for \eqref{def:Jdelta}: $J_\Delta(\alpha,\beta)=J_b(\alpha,\beta)\times
\frac{4a}{3}\left(\frac{a^3-1}{a-1}\right)\left(\frac{ln(a)}{a^2-1}\right)^2$. We remark that $J_b(\alpha,\beta)\leq I_b(\alpha,\beta),\forall T>0$. So, as expected, $\forall \Delta >0, \; J_\Delta(\alpha,\beta)\leq I_\Delta(\alpha,\beta)$. Hence, contrast estimation with prior knowledge on the model  multiplicativity (see Section \ref{multiplicative}) leads to a more accurate confidence interval than the general case with no available information on $\beta$.
\subsection{ A two factor model}
We consider here the same example as \cite{glo09} (see e.g. \cite{lon95}). Let us define $X_t=(Y_t,R_t)$ as the solution on $[0,1]$ of
\begin{equation}\label{exglo}
	\begin{array}{l}
		dY_t=\left(R_t+\mu_1\right)dt+\epsilon\kappa_1dB^1_t, \; Y_0=y_0\in \R \\
dR_t=\mu_2\left(m-R_t\right)dt+\epsilon \kappa_2 \sqrt{R_t}\left(\rho dB^1_t+\sqrt{1-\rho^2}dB^2_t\right), \; R_0=r_0>0.	\end{array}
\end{equation}
Hence, we get that $\Sigma((\kappa_1,\kappa_2,\rho),(Y_t,R_t))=\begin{pmatrix}
                                                                \kappa_1^2 & \kappa_1\kappa_2\rho\sqrt{R_t}\\
\kappa_1\kappa_2\rho\sqrt{R_t}& \kappa_2^2 R_t
                                                               \end{pmatrix}$.\\
For $r_0\neq m$ the diffusion process satisfies {\bf(S1)-(S5)} and we can estimate parameters $\alpha=(\mu_1,\mu_2,m)$ and $\beta=(\kappa_1^2,\kappa_2^2,\rho)$ with our minimum contrast estimators defined in \eqref{est:cls}, \eqref{alphatild} and \eqref{def:estimateurs_chech}.
As \cite{glo09}, we investigate the case of $\mu_1=\mu_2=m=\kappa_1=\kappa_2=1$, $\rho=0.3$, and $(y_0,r_0)=(0,1.5)$, for two values of $\epsilon$, $0.1$ and $0.01$. Similarly, we present in Tables \ref{table1} and \ref{table2} contrast estimators (empirical means and standard deviations) over 400 runs of the diffusion process \eqref{exglo} simulated based on an Euler scheme. For each of these simulations, different values of the number of observations $n$ are used ($n=10,20,50,100$ observations) to infer parameters. Results of Gloter and Sorensen were reproduced using their contrast based on an expansion at order 2 of the function defined in Section 2.3. in \cite{glo09}.\\
For $\epsilon=0.01$, Table \ref{table1} results are very similar to those in \cite{glo09}.
When $\epsilon=0.1$, we can distinguish two different patterns. For $\bar{\alpha}$ and $(\check{\alpha},\check{\beta})$ results exhibit a lack of accuracy on $\mu_2$, similarly to those of \cite{glo09}. The second pattern concerns $\tilde{\alpha}$, where the bias on $\mu_2$ (for $n=10,20,50$), less important than for $\bar{\alpha}$ and $(\check{\alpha},\check{\beta})$, is partially balanced by an increase in the uncertainty of $m$. These results show that prior knowledge on the model (more specifically fixing the diffusion parameters to their true value) leads to a different behaviour of the estimator. From a theoretical point of vue, this is explained only by the shape of $S_k^{\alpha,\beta_0}$, which does not consider equal weights for all observations.\\
In addition, a decrease in accuracy is obtained when increasing the number of observations. This could be explained by the behaviour of $N_k(X,\alpha)$ which depends on the variation of slope between two consecutive data points. Indeed, in this particular model \eqref{exglo}, where the drift is almost linear (and hence the local gradient close to zero), variations of local slopes increase with the number of observations randomly distributed around the global slope. When performing the estimation on a longer time interval with the same number of observations ($[0,5]$, $n=50$), the decrease in accuracy following the increase in the number of observations is partially counterbalanced. 
\begin{table}[!ht]
\label{table1}
{\footnotesize
\newcommand{\mc}[3]{\multicolumn{#1}{#2}{#3}}
\begin{tabular}{|c|c|ll|c|ll|c|lllll}\hline
$\epsilon=0.01$ & \mc{3}{c|}{$\bar{\alpha}$ ($\beta$ unknown)} & \mc{3}{c|}{$\tilde{\alpha}$ ($\beta=\beta_0$ fixed)} & \mc{6}{c|}{($\check{\alpha},\check{\beta}$) (Small Delta)}\\\hline
n & $\bar{\mu}_1$ & \mc{1}{c|}{$\bar{\mu}_2$} & \mc{1}{c|}{$\bar{m}$} & $\tilde{\mu}_1$ & \mc{1}{c|}{$\tilde{\mu}_2$} & \mc{1}{c|}{$\tilde{m}$} & $\check{\mu}_1$ & \mc{1}{c|}{$\check{\mu}_2$} & \mc{1}{c|}{$\check{m}$} & \mc{1}{c|}{$\check{\kappa}^2_1$} & \mc{1}{c|}{$\check{\kappa}^2_2$} & \mc{1}{c|}{$\check{\rho}$}\\\hline
n=10 & 1.001 & \mc{1}{c|}{1.007} & \mc{1}{c|}{0.996} & 1.000 & \mc{1}{c|}{1.005} & \mc{1}{c|}{0.997} & 1.000 & \mc{1}{c|}{1.013} & \mc{1}{c|}{0.999} & \mc{1}{c|}{0.971} & \mc{1}{c|}{0.728} & \mc{1}{c|}{0.328}\\
 & (0.01) & \mc{1}{c|}{(0.13)} & \mc{1}{c|}{(0.04)} & (0.01) & \mc{1}{c|}{(0.12)} & \mc{1}{c|}{(0.04)} & (0.01) & \mc{1}{c|}{(0.13)} & \mc{1}{c|}{(0.04)} & \mc{1}{c|}{(0.45)} & \mc{1}{c|}{(0.37)} & \mc{1}{c|}{(0.33)}\\\hline
n=20 & 1.000 & \mc{1}{c|}{1.012} & \mc{1}{c|}{0.999} & 1.000 & \mc{1}{c|}{1.003} & \mc{1}{c|}{0.997} & 1.000 & \mc{1}{c|}{1.012} & \mc{1}{c|}{0.999} & \mc{1}{c|}{0.973} & \mc{1}{c|}{0.853} & \mc{1}{c|}{0.306}\\
 & (0.01) & \mc{1}{c|}{(0.12)} & \mc{1}{c|}{(0.04)} & (0.01) & \mc{1}{c|}{(0.12)} & \mc{1}{c|}{(0.04)} & (0.01) & \mc{1}{c|}{(0.12)} & \mc{1}{c|}{(0.04)} & \mc{1}{c|}{(0.32)} & \mc{1}{c|}{(0.29)} & \mc{1}{c|}{(0.23)}\\\hline
n=50 & 1.000 & \mc{1}{c|}{1.012} & \mc{1}{c|}{0.999} & 1.000 & \mc{1}{c|}{1.000} & \mc{1}{c|}{0.996} & 1.000 & \mc{1}{c|}{1.012} & \mc{1}{c|}{0.999} & \mc{1}{c|}{0.982} & \mc{1}{c|}{0.910} & \mc{1}{c|}{0.302}\\
 & (0.01) & \mc{1}{c|}{(0.12)} & \mc{1}{c|}{(0.04)} & (0.01) & \mc{1}{c|}{(0.12)} & \mc{1}{c|}{(0.04)} & (0.01) & \mc{1}{c|}{(0.12)} & \mc{1}{c|}{(0.04)} & \mc{1}{c|}{(0.20)} & \mc{1}{c|}{(0.19)} & \mc{1}{c|}{(0.14)}\\\hline
n=100 & 1.000 & \mc{1}{c|}{1.013} & \mc{1}{c|}{0.999} & 1.000 & \mc{1}{c|}{0.995} & \mc{1}{c|}{0.995} & 1.000 & \mc{1}{c|}{1.011} & \mc{1}{c|}{0.999} & \mc{1}{c|}{1.001} & \mc{1}{c|}{0.953} & \mc{1}{c|}{0.310}\\
 & (0.01) & \mc{1}{c|}{(0.12)} & \mc{1}{c|}{(0.04)} & (0.01) & \mc{1}{c|}{(0.12)} & \mc{1}{c|}{(0.04)} & (0.01) & \mc{1}{c|}{(0.12)} & \mc{1}{c|}{(0.04)} & \mc{1}{c|}{(0.14)} & \mc{1}{c|}{(0.14)} & \mc{1}{c|}{(0.09)}\\\hline
\end{tabular}
}
\caption{Mean (standard deviation) of minimum contrast estimators for parameters of  \eqref{exglo} based on 400 simulated trajectories with $\mu_1=\mu_2=m=\kappa_1=\kappa_2=1$, $\rho=0.3$, using $\epsilon=0.01$ and $n=10,20,50,100$.}
\end{table}
\begin{table}[!ht]
\label{table2}
{\footnotesize
\newcommand{\mc}[3]{\multicolumn{#1}{#2}{#3}}
\begin{tabular}{|c|c|ll|c|ll|c|lllll}\hline
$\epsilon=0.1$ & \mc{3}{c|}{$\bar{\alpha}$ ($\beta$ unknown)} & \mc{3}{c|}{$\tilde{\alpha}$ ($\beta=\beta_0$ fixed)} & \mc{6}{c|}{($\check{\alpha},\check{\beta}$) (Small Delta)}\\\hline
n & $\bar{\mu}_1$ & \mc{1}{c|}{$\bar{\mu}_2$} & \mc{1}{c|}{$\bar{m}$} & $\tilde{\mu}_1$ & \mc{1}{c|}{$\tilde{\mu}_2$} & \mc{1}{c|}{$\tilde{m}$} & $\check{\mu}_1$ & \mc{1}{c|}{$\check{\mu}_2$} & \mc{1}{c|}{$\check{m}$} & \mc{1}{c|}{$\check{\kappa}^2_1$} & \mc{1}{c|}{$\check{\kappa}^2_2$} & \mc{1}{c|}{$\check{\rho}$}\\\hline
n=10 & 1.000 & \mc{1}{c|}{1.723} & \mc{1}{c|}{0.892} & 1.005 & \mc{1}{c|}{1.052} & \mc{1}{c|}{0.667} & 0.998 & \mc{1}{c|}{1.678} & \mc{1}{c|}{0.997} & \mc{1}{c|}{0.927} & \mc{1}{c|}{0.769} & \mc{1}{c|}{0.422}\\
 & (0.10) & \mc{1}{c|}{(1.23)} & \mc{1}{c|}{(0.43)} & (0.10) & \mc{1}{c|}{(0.92)} & \mc{1}{c|}{(0.49)} & (0.10) & \mc{1}{c|}{(1.23)} & \mc{1}{c|}{(0.41)} & \mc{1}{c|}{(0.43)} & \mc{1}{c|}{(0.36)} & \mc{1}{c|}{(0.23)}\\\hline
n=20 & 1.001 & \mc{1}{c|}{1.754} & \mc{1}{c|}{0.922} & 1.011 & \mc{1}{c|}{0.930} & \mc{1}{c|}{0.590} & 1.000 & \mc{1}{c|}{1.718} & \mc{1}{c|}{0.930} & \mc{1}{c|}{0.966} & \mc{1}{c|}{0.864} & \mc{1}{c|}{0.344}\\
 & (0.10) & \mc{1}{c|}{(1.24)} & \mc{1}{c|}{(0.40)} & (0.10) & \mc{1}{c|}{(0.90)} & \mc{1}{c|}{(0.51)} & (0.10) & \mc{1}{c|}{(1.20)} & \mc{1}{c|}{(0.39)} & \mc{1}{c|}{(0.29)} & \mc{1}{c|}{(0.29)} & \mc{1}{c|}{(0.18)}\\\hline
n=50 & 1.000 & \mc{1}{c|}{1.760} & \mc{1}{c|}{0.928} & 1.029 & \mc{1}{c|}{0.509} & \mc{1}{c|}{0.342} & 1.001 & \mc{1}{c|}{1.82} & \mc{1}{c|}{0.994} & \mc{1}{c|}{0.971} & \mc{1}{c|}{0.832} & \mc{1}{c|}{0.167}\\
 & (0.10) & \mc{1}{c|}{(1.23)} & \mc{1}{c|}{(0.40)} & (0.10) & \mc{1}{c|}{(0.70)} & \mc{1}{c|}{(0.61)} & (0.10) & \mc{1}{c|}{(1.18)} & \mc{1}{c|}{(0.31)} & \mc{1}{c|}{(0.09)} & \mc{1}{c|}{(0.08)} & \mc{1}{c|}{(0.07)}\\\hline
n=100 & 1.001 & \mc{1}{c|}{1.778} & \mc{1}{c|}{0.933} & 1.051 & \mc{1}{c|}{0.122} & \mc{1}{c|}{0.410} & 1.000 & \mc{1}{c|}{1.825} & \mc{1}{c|}{0.987} & \mc{1}{c|}{0.979} & \mc{1}{c|}{0.846} & \mc{1}{c|}{0.156}\\
 & (0.10) & \mc{1}{c|}{(1.23)} & \mc{1}{c|}{(0.40)} & (0.10) & \mc{1}{c|}{(0.27)} & \mc{1}{c|}{(1.22)} & (0.10) & \mc{1}{c|}{(1.19)} & \mc{1}{c|}{(0.33)} & \mc{1}{c|}{(0.07)} & \mc{1}{c|}{(0.06)} & \mc{1}{c|}{(0.05)}\\\hline
\end{tabular}
}
\caption{Mean (standard deviation) of minimum contrast estimators for parameters of  \eqref{exglo} based on 400 simulated trajectories with $\mu_1=\mu_2=m=\kappa_1=\kappa_2=1$, $\rho=0.3$, using $\epsilon=0.1$ and $n=10,20,50,100$.}
\end{table}
\FloatBarrier
\subsection{Epidemic models and data}
\label{epidemics}
Here, we present an example where $\epsilon$, corresponding to the normalizing constant $1/\sqrt N$  has an intrinsic meaning. 
One of the simplest models for the study of epidemic spread is the $SIR$ (Susceptible-Infectious-Removed from the infectious chain) model, where each individual can find himself at a given time in one of these three mutually exclusive health states.\\
One classical representation of the $SIR$ model in closed population is the bi-dimensional continuous-time Markovian jump process: $X_t=(S_t,I_t)$ with initial state $X_0=(N-m,m)$ and transitions $(S,I)\stackrel{\frac{\lambda}{N}SI}{\longrightarrow}(S-1,I+1)$ and $(S,I)\stackrel{\gamma I}{\longrightarrow}(S,I-1)$.\\
The normalization of this process  based  on the population size $N$ asymptotically leads to an ODE system: $x(t)=(s(t),i(t),r(t)=1-s(t)-i(t))$, with $x(0)=(1-m/N, m/N,0)$, which is solution of (\ref{deter}) for $b((\lambda,\gamma),x)=\begin{pmatrix}
 -\lambda x_1 x_2 \\
 \lambda x_1 x_2-\gamma x_2                                                                                                                                                                                                                                                                                                                                                                                                                                                                                                                                                                                                                                                                                                                                                                                                                                                                                                                                                                                                                                                                                                                                                                                                                                                                                                                                                                                                                                                                                                        \end{pmatrix}.$\\
Before passing to the limit, by defining $\Sigma((\lambda,\gamma),x)=\begin{pmatrix}
\lambda x_1 x_2 & -\lambda x_1 x_2 \\
-\lambda x_1 x_2& \lambda x_1 x_2+\gamma x_2                                                                                                                                                                                                                                                                                                                                                                                                                                                                                                                                                                                                                                                                                                                                                                                                                                                                                                                                                                                                                                                                                                                                                                                                                                                                                                                                                                                                                                                                                                        \end{pmatrix}$, we can write the infinitesimal generator of the renormalized Markovian jump process ($X(t)/N$) as the solution of \\
{\footnotesize$\mathcal{A}_{N}\left(f(x)\right)=N\lambda x_1x_2\left( f(x_1-\frac{1}{N},x_2+\frac{1}{N})-f(x_1,x_2)\right)+N\gamma x_2\left( f(x_1,x_2-\frac{1}{N})-f(x_1,x_2)\right)$}. 
We also have $\mathcal{A}_{N}\left(f(x)\right)=\mathcal{A}^{(2)}_{N}\left(f(x)\right)+\mathcal{A}^{(3+)}_{N}\left(f(x)\right)$, with \\
$\mathcal{A}^{(2)}_{N}\left(f(x)\right)=b((\lambda,\gamma),x)\bigtriangledown f(x) +\frac{1}{2N}\som{i,j=1}{2}\deriv{^2f}{x_i\partial x_j}(x)\Sigma((\lambda,\gamma),x)_{i,j}$ and where $\mathcal{A}^{(3+)}_{N}$ contains all the derivatives terms of order 3 and above. Then, approximating the renormalized Markovian jump process by a Markov process with generator $\mathcal{A}^{(2)}_{N}$,  leads to a diffusion process $X_t=(s_t,i_t)$ with drift $b$ and diffusion matrix $\Sigma$, which can be rewritten as the solution of:
\begin{equation}\label{model_SIR}
\begin{array}{ccl}
ds_t&=&-\lambda s_ti_t dt+\frac{1}{\sqrt{N}}\sqrt{\lambda s_ti_t}dB_1(t)\\
\par \\
di_t&=&(\lambda s_ti_t-\gamma i_t)dt-\frac{1}{\sqrt{N}}\sqrt{\lambda s_ti_t}dB_1(t)+\frac{1}{\sqrt{N}}\sqrt{\gamma i_t}dB_2(t).\\
\end{array}
\end{equation}
Here, $\lambda$ and $\gamma$ represent transmission and recovery rates, respectively and are the two parameters to be estimated.\\
So, system (\ref{model_SIR}) can naturally be viewed as a diffusion with a small diffusion coefficient ($\epsilon =N^{-1/2}$). Moreover, parameters to be estimated are both in drift and diffusion coefficients, with the specificity that $\alpha=(\lambda,\gamma)=\beta$ (with the notations of (\ref{model})). Besides, since epidemics are discretely observed, the statistical setting is defined by data on a fixed interval $[0,T]$, at times $t_k=k \Delta$, with $T=n \Delta$ and $n$ the number of data points ($\Delta$ not necessarily small).\\
The performances of our method for epidemic models in the case of a fixed sampling interval $\Delta$ and for $\Delta\rightarrow 0$ were evaluated on discretized exact simulated trajectories of the pure Markov jump process $X_t$ and compared to estimators provided by the method of \cite{glo09}. We considered the Maximum Likelihood Estimator (MLE) \cite{and00} of the Markov Jump process, built using all the jumps, as the reference. Simulated data were generated by using the Gillespie algorithm \cite{gil77} after specifying $(N,m,\lambda,\gamma)$. Two population sizes were considered $N\in [100;10000]$ and $m/N$ was set to $0.01$ for all simulations. $(\lambda,\gamma)$ were chosen such that their ratio takes a realistic value. Indeed, $\lambda / \gamma$ defines for the $SIR$ model used here a key parameter in epidemiology, $R_0$, which represents the mean number of secondary infections generated by a primary case in a totally susceptible population. 
We have chosen  $R_0=1.2, \gamma=1/3 \;(days^{-1})$ (and hence $\lambda=0.4 \;(days^{-1})$) to represent a realistic scenario (parameter values close to influenza epidemics). We considered $T=50$ days, in order to capture the pattern of an average trajectory for the case $N=100$ (shorter epidemic duration than for $N=10000$). Three values of $n$ were tested: $10,50,100$. For each simulated scenario, means and theoretical confidence intervals ($95\%$) for $\lambda$ and $\gamma $ were calculated on 1000 runs for each parameter and for each estimation method.\\
Figures \ref{fig:N10000R12}  and \ref{fig:N100R12} summarize numerical results (only drift estimators are provided). According to our findings, contrast based estimators are very effective even for a few amount of observations, compared with the MLE. As expected, for all scenarios, we can see an improvement in the accuracy as the number of observations increases for estimators $\bar{\alpha}$, $\check{\alpha}$ and the estimator of \cite{glo09}. On the contrary, $\tilde{\alpha}$-estimators accuracy decreases as the number of observations increases. This phenomenon is due to the shape of $S_k^{\alpha,\beta}$ (defined in \eqref{S_k+1}), which confers greater weigths to the beginning and the end of data (as for the two factor model \eqref{exglo} above). For $N=10000$, it is important to notice that the bias is quite negligible from an epidemiological point of view. Indeed, the bias for $1/\gamma$ has an order of magnitude of one hour whereas an accuracy of one day would be acceptable. For  the case $N=100$, only emerging trajectories were considered, based on an epidemiological relevant criteria (epidemic size above $10\%$ of the population size). We can remark that MLE provides less satisfactory estimations for $\gamma$. Our contrast estimators for $n=100$ perform globally well, except for $\tilde{\alpha}$. But even in this last case, contrary to the MLE, the ratio $\lambda/\gamma$ is close to the true value despite a bias on both $\lambda$ and $\gamma$ separately.  
\begin{figure}[!ht]
 \includegraphics[width=0.99\textwidth]{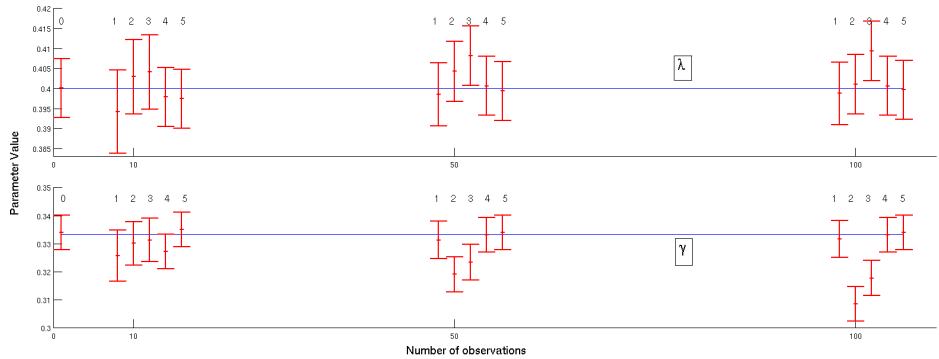}
 \caption{\label{fig:N10000R12} Mean values and theoretical CI ($95\%$) of the estimators of $\lambda$ and $\gamma$. Labels are 0: MLE (with all data available), 1: $\bar{\alpha}$ ($\beta$ unknown), 2: $\tilde{\alpha}$ ($\beta_0 $ known), 3: $\tilde{\alpha}$ ($\beta=\alpha$) , 4: $\check{\alpha}$ (small $\Delta$) and 5: the estimator of drift parameters in \cite{glo09}). Results based on 1000 runs for $N=10000$, $T=50$, $\lambda=0.4$, $\gamma=1/3$ and for $n=10,50,100$.}
\end{figure}
\begin{figure}[!ht]
 \includegraphics[width=0.99\textwidth]{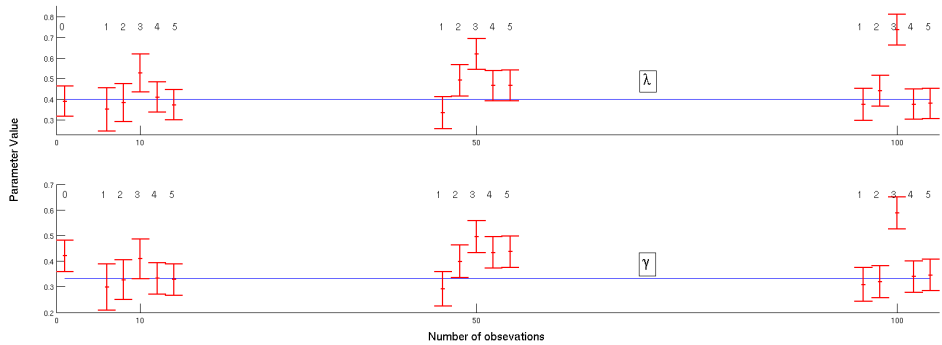}
 \caption{\label{fig:N100R12} Mean values and theoretical CI ($95\%$) of the estimators of $\lambda$ and $\gamma$. Labels are 0: MLE (with all data available), 1: $\bar{\alpha}$ ($\beta$ unknown), 2: $\tilde{\alpha}$ ($\beta_0 $ known), 3: $\tilde{\alpha}$ ($\beta=\alpha$) , 4: $\check{\alpha}$ (small $\Delta$) and 5: the estimator of drift parameters in \cite{glo09}). Results based on 1000 runs for $N=100$, $T=50$, $\lambda=0.4$, $\gamma=1/3$ and for $n=10,50,100$.}
\end{figure}
\FloatBarrier

Our results are promising in the epidemiological context, since the minimum contrast estimators are both accurate and not computationally expensive, even for very noisy data ($N=100$). Ongoing research is devoted to the extension of these findings to the more realistic case of partially observed epidemic data.

\section{Aknowledgments}
Partial financial support for this research was provided by Ile de France Regional Council under MIDEM
project in the framework DIM Malinf. 

\bibliographystyle{elsarticle-harv.bst}

\newpage

\section{Appendix}
\subsection{Some useful analytical properties}
\label{bound:analytics}
We state here a series of regularity properties of $(\alpha,t)\rightarrow \Phi_\alpha(t,t_0)$ and $x_\alpha(t)$.
Let us first consider $\Phi_\alpha$. A Taylor expansion of $t\rightarrow \Phi_\alpha(t,t_{k-1})$ yields using \eqref{eqphi}\\
$\Phi_\alpha(t_k,t_{k-1})=I_p+\Delta \deriv{b}{x}(\alpha,x_\alpha(t_{k-1})+\Delta r(\alpha,t_{k-1},\Delta)$ where $r(\alpha,t_{k-1},\Delta)$ converges uniformly to 0 on $[0,T]\times K_a$. Hence, 
\begin{equation}\label{bound:phi-id}
\abs{ \frac{1}{\Delta}\left(\fy{k}{k-1}-I_p\right)-\deriv{b}{x}(\alpha,x_\alpha(t_{k-1}))}\tend{\Delta}{0}0.
\end{equation}
As a consequence $(\alpha,t)\rightarrow \Phi_\alpha(t,t_0)$ is uniformly bounded on $K_a\times [0,T]$.
Consider now the properties of $\alpha\rightarrow \Phi_\alpha(t,t_0)$.
\begin{lemma}
\label{rem:Phi}
 Under the assumption that $b(\alpha,x)\in C^3(K_a\times U)$, the function $\alpha\rightarrow \Phi_\alpha(t,t_0)$ is in $C^2(K_a)$.
\end{lemma}
\begin{proof}
 Classically, we just prove here that $\Phi_\alpha$ is continuous w.r.t. $\alpha$ if $\alpha\rightarrow\deriv{b}{x}(\alpha,x_\alpha(t))$ is continuous.
Set $M_h(t)=\Phi_{\alpha+h}(t,t_0)-\Phi_{\alpha}(t,t_0)$. Using \eqref{eqphi}, we have\\
 $M_h(t)=\integ{t_0}{t}\deriv{b}{x}(\alpha+h,x_{\alpha+h}(s))M_h(s)ds+\integ{t_0}{t}\left(\deriv{b}{x}(\alpha+h,x_{\alpha+h}(s))-\deriv{b}{x}(\alpha,x_{\alpha}(s))\right)\Phi_\alpha(s,t_0)ds.$
By \eqref{bound:phi-id} and the continuity of $t\rightarrow \Phi_\alpha(t,t_0)$, we can define $K_0=\super{K_a\times[0,T]}\norm{\Phi_\alpha(t,t_0)}$ and $K= \super{K_a\times[0,T]}\norm{\deriv{b}{x}(\alpha,x_{\alpha}(t))}$. Setting $\gamma_h(t)=K_0\integ{t_0}{t}\norm{\deriv{b}{x}(\alpha+h,x_{\alpha+h}(s))-\deriv{b}{x}(\alpha,x_{\alpha}(s))}ds$, we have $\norm{M_h(t)}\leq \gamma_h(t)+K\integ{0}{t}\norm{M_h(s)}ds$. Applying Gronwall's inequality to $\norm{M_h}$ yields $\norm{M_h(t)}\leq \gamma_h(t)+K\integ{t_0}{t}\gamma_h(s) e^{K(t-s)}ds$. By the Lebesgue dominated convergence theorem, $\gamma_h(t)$ goes to 0 as $h\rightarrow 0$, which implies the same property for $M_h(t)$.
\end{proof}
The existence of derivatives for $\Phi_\alpha$ w.r.t. $\alpha$ are obtained similarly. Moreover, expanding $\deriv{\Phi_\alpha}{\alpha_i}$, $\deriv{^2 \Phi_\alpha}{\alpha_i \partial \alpha_j}$ in Taylor series at point $t_{k-1}$, they satisfy
\begin{equation}\label{bound:dphi/dalpha}\forall i\leq a, \; \abs{\frac{1}{\Delta}\deriv{\fy{k}{k-1}}{\alpha_i}(\alpha_0)-\deriv{^2 b(\alpha,x_\alpha(t_{k-1}))}{x\partial \alpha_i}(\alpha_0,\xzero{k-1})}\tend{\Delta}{0}0,\end{equation}
\begin{equation}\label{bound:d2phi/dalpha2}\forall i,j\leq a, \; \abs{\frac{1}{\Delta}\deriv{^2\fy{k}{k-1}}{\alpha_i\partial\alpha_j}(\alpha_0)-\deriv{^3b(\alpha,x_\alpha(t_{k-1}))}{x\partial \alpha_i\partial \alpha_j}(\alpha_0,x_{\alpha_0}(t_{k-1})))}\tend{\Delta}{0}0\end{equation}
and all left terms are bounded as $\Delta\rightarrow 0$.\\
Let us now consider $x_\alpha$ and its derivatives. Using (\ref{deter}), and expanding $t\rightarrow x_\alpha(t)$ in Taylor series at point $t_{k-1}$, as above, yields
\begin{equation}\label{bound:x}
 \abs{\frac{1}{\Delta}\left(x_\alpha(t_k)-x_\alpha(t_{k-1})\right)-b(\alpha,x_\alpha(t_{k-1}))}\tend{\Delta}{0}0,
\end{equation}
\begin{equation}\label{bound:dx/dalpha}\abs{\frac{1}{\Delta}\left(\deriv{\x{k}}{\alpha_i}-\deriv{\x{k-1}}{\alpha_i}\right)(\alpha_0)-\deriv{b(\alpha,x_\alpha(t_{k-1}))}{\alpha_i}(\alpha_0)}\tend{\Delta}{0}0,
\end{equation}
\begin{equation}\label{bound:d2x/dalpha2}\abs{\frac{1}{\Delta}\left(\deriv{^2\x{k}}{\alpha_i\partial\alpha_j}(\alpha_0)-\deriv{^2\x{k-1}}{\alpha_i\partial\alpha_j}(\alpha_0)\right)-\deriv{b(\alpha,x_\alpha(t_{k-1}))}{\alpha_i\partial \alpha_j}(\alpha_0)}\tend{\Delta}{0}0.\end{equation}
\subsection{ Proof of Corollary \ref{Rteps}}
\label{ProofRteps}
The proof of (i) is given in \cite{frei84} (Theorem 2.2) but we need a more refined result on the increments of $\Reste{}$. For sake of clarity, we omit in the sequel $\theta$ and $\alpha$ (and therefore denote $\deriv{f}{x}(x_0)$ by $f'(x_0)$), and we denote by $\norm{}$ either a norm on $\R^p$ or on $M_p(\R)$. We study successively $\Restun{}$ and $\Reste{}$.\\
 Using $X^\epsilon_{t}=x(t)+\epsilon R^{1,\epsilon}(t)$ and \eqref{model}, $R^{1,\epsilon}$ satisfies\\
$R^{1,\epsilon}(t)=\integ{0}{t} \frac{1}{\epsilon}\left( b(x(s)+\epsilon R^{1,\epsilon}(s))-b(x(s))\right)ds+\integ{0}{t} \sigma(x(s)+\epsilon R^{1,\epsilon}(s))dB_s, \; R^{1,\epsilon}(0)=0.$\\
 Hence, $R^{1,\epsilon}$ satisfies a stochastic differential equation with drift $d_\epsilon(t,z)$ and diffusion coefficient $v_\epsilon(t,z)$ where $d_\epsilon(t,z)=\frac{1}{\epsilon}\left( b(x(t)+z)-b(x(t))\right)$ and $v_\epsilon(t,z)=\sigma(x(t)+\epsilon z)$.\\
 Using that the derivatives of $b$ and $\sigma$ are uniformly bounded on U, these two coefficients satisfy \\
$\norm{d_\epsilon(t,z)-d_\epsilon(t,z')}\leq \super{x\in U}\norm{b'(x)}\norm{z-z'}$, 
$\norm{v_\epsilon(t,z)-v_\epsilon(t,z')}\leq \super{x\in U}\norm{\sigma'(x)}\epsilon\norm{z-z'}$\\
and $\norm{d_\epsilon(t,z)}^2+\norm{v_\epsilon(t,z)}^2\leq C_1(1+\norm{z}^2)$ where $C_1=max(\super{x\in U}\norm{b'(x)},\super{x\in U}\norm{\sigma'(x)}).$ \\ Hence, $X_1$ is finite and independent of $\epsilon$. An application of Theorem 2.9 of \cite{kar91} yields that there is  $C$ a constant depending only on $C_1$ and $T$ such that $\forall t\leq T , \;\E{\norm{R^{1,\epsilon}(t)}^2}\leq C e^{Ct}.$\\
Let us now study $R^{2,\epsilon}(t)$. Using \eqref{taylor_stochastique} we get,\\
\begin{equation}\label{sdeR2} R^{2,\epsilon}(t)=\integ{0}{t}\tilde{d}_\epsilon(s,\omega,R^{2,\epsilon}(s))ds+\integ{0}{t}\tilde{v}_\epsilon(s,\omega)dB_s, \; R^{2,\epsilon}(0)=0,\end{equation} with
$\tilde{d}_\epsilon(s,\omega,z)=\frac{1}{\epsilon^2}\left( b(x(t)+\epsilon g(t,\omega) +\epsilon^2 z) -b(x(t)) -\epsilon b'(x(t)) g(t,\omega)\right)$,\\ $\tilde{v}_\epsilon(s,\omega)=\frac{1}{\epsilon}\left( \sigma(x(t)+ \epsilon R^{1,\epsilon}(t,\omega))-\sigma(x(t))\right)$.\\
First, let us check that the stochastic integral above is well defined. For this, we compute  $\E{\norm{\integ{0}{t}\tilde{v}_\epsilon(s,\omega) \trans{\tilde{v}}_\epsilon(s,\omega)ds}}$. Applying a Taylor expansion to $\sigma(x(s))$ yields $\tilde{v}_\epsilon(s,\omega)=\left(\integ{0}{1}\sigma'(x(s)+u\epsilon R^{1,\epsilon}(s))du\right)R^{1,\epsilon}(s)$. Hence, $\norm{\tilde{v}_\epsilon(s,\omega)}\leq \super{x\in U}\norm{\sigma'(x)}\norm{R^{1,\epsilon}(s)}$ and\\ $\E{\norm{\integ{0}{t}\tilde{v}_\epsilon(s,\omega) \trans{\tilde{v}}_\epsilon(s,\omega)ds}}\leq \super{x\in U}\norm{\sigma'(x)}^2 \integ{0}{t}\E{\norm{R^{1,\epsilon}(s)}^2}ds\leq C_1^2 \frac{e^{Ct}-1}{C}$.\\ Consider now the drift term $\tilde{d}_\epsilon(s,\omega,z)$. A Taylor expansion with integral remainder yields \\
$\begin{array}{rl}
\epsilon^2 \tilde{d}_\epsilon(s,\omega,z)=&\left(b(x(t)+\epsilon g(t) +\epsilon^2 z)-b(x(t)+\epsilon g(t)\right)+\left( b(x(t)+\epsilon g(t))-b(x(t)) -\epsilon b'(x(t)) g(t)\right)\\
=&\epsilon^2 \left(\integ{0}{1} b'(x(t)+\epsilon g(t) + u \epsilon^2 z)du \right) z+ \epsilon^2\trans{g(t)} \left( \integ{0}{1} (1-u) b''(x(t)+u\epsilon g(t) )du\right) g(t).
\end{array}$\\
Hence, $\tilde{d}_\epsilon(s,\omega,z)$ is bounded independently of $\epsilon$ by\\
 $\norm{\tilde{d}_\epsilon(s,\omega,z)}\leq \super{x\in U}{b'(x)} \norm{z} + \super{x\in U} \norm{b''(x)} \norm{g(t)}^2$. Now, using \eqref{sdeR2}, we get\\
 $\norm{R^{2,\epsilon}(t)}^2\leq 2\left(\norm{\integ{0}{t}\tilde{d}_\epsilon(s,\omega,R^{2,\epsilon}(s))ds}^2+\norm{\integ{0}{t}\tilde{v}_\epsilon(s,\omega)dB_s}^2\right)$. We already prove that the last term above has a finite expectation. It remains to study the first term.\\ $\E{\norm{R^{2,\epsilon}(t)}^2}\leq 2C_1^2\integ{0}{t} \E{\norm{R^{2,\epsilon}(s)}^2}ds+ H(t)$ with \\$H(t)=2 \super{x \in U}\norm{b''(x)}^2 \integ{0}{t}\E{\norm{g(t)}^4}ds +C_1^2 \frac{e^{Ct}-1}{C}$. Applying Gronwall's inequality yields
$\E{\norm{R^{2,\epsilon}(t)}^2}\leq H(t)+ 2C_1^2\integ{0}{t} H(s) e^{2C_1^2(t-s)}ds$. Since $g(s)$ is a continous Gaussian process, $\super{s\in[0,T]}\E{\norm{g(s)}^4}$ is finite, and $\abs{H(t)}\leq Kt$, so that $\E{\norm{R^{2,\epsilon}(t)}^2}\leq K' t$ with $K'=K(1+2C_1^2)$.\\
Consider now (ii), $R^{2,\epsilon}(t+h)-R^{2,\epsilon}(t)=\integ{t}{t+h}\tilde{d}_\epsilon(s,\omega,R^{2,\epsilon}(s))ds+\integ{t}{t+h}\tilde{v}_\epsilon(s,\omega)dB_s.$\\
$\E{\norm{R^{2,\epsilon}(t+h)-R^{2,\epsilon}(t)}^2}=\Ec{\E{\norm{R^{2,\epsilon}(t+h)-R^{2,\epsilon}(t)}^2}}{}=\E{\mathbb{E}_{X_t}\left[\norm{R^{2,\epsilon}(h)}^2\right]}.$\\
By the Markov property of $X_t$ we get that $\mathbb{E}_{X_t}\left[\norm{R^{2,\epsilon}(h)}^2\right]\leq K'h$.
\subsection{ Proof of Proposition \ref{prop:cls}}
\label{proof:cls}

Let us first prove (i). The  processes $\bar{U}_{\epsilon,\Delta}(\alpha,(X_{t_k}(\omega)))$ are almost surely continuous with continuity modulus
\begin{equation*}
w(\bar{U}_{\epsilon,\Delta},\eta)=sup\{
\abs{\bar{U}_{\epsilon,\Delta}(\alpha,\cdot)-\bar{U}_{\epsilon,\Delta}(\alpha',\cdot)},
 (\alpha,\alpha') \in
\bar{K}_a^2, \norm{\alpha-\alpha'}\leq \eta\}.
\end{equation*}
We have $\abs{\bar{U}_{\epsilon,\Delta}(\alpha,(X_{t_k}))-\bar{U}_{\epsilon,\Delta}(\alpha',(X_{t_k}))}\leq 
\abs{\bar{U}_{\epsilon,\Delta}(\alpha,(X_{t_k}))-\bar{U}_{\epsilon,\Delta}(\alpha,(x_{\alpha_0}(t_k)))}+\abs{
\bar{U}_{\epsilon,\Delta}(\alpha',(X_{t_k}))-\bar{U}_{\epsilon,\Delta}(\alpha',(x_{\alpha_0}(t_k)))}+\abs{
\bar{K}_\Delta(\alpha_0,\alpha)-\bar{K}_\Delta(\alpha_0,\alpha')}$.\\
Using  formula (\ref{tsto1}),
 $N_k(X,\alpha)-N_k(x_{\alpha_0},\alpha)=\epsilon R^{1,\epsilon}_{\theta_0}(t_k)-\Phi_\alpha(t_k,t_{k-1})\epsilon
R^{1,\epsilon}_{\theta_0}(t_{k-1})$ and\\
 $\begin{array}{rl}\abs{\bar{U}_{\epsilon,\Delta}(\alpha,(X_{t_k}))-\bar{U}_{\epsilon,\Delta}(\alpha,(x_{\alpha_0}
(t_k)))}\leq&
\frac{1}{\Delta}\som{k=1}{n}\norm{N_k(X,\alpha)-N_k(x_{\alpha_0},\alpha)}\norm{N_k(x_{\alpha_0},\alpha)+N_k(X,\alpha)}\\
   \leq&\frac{2n}{\Delta}\underset{t\in[0,T]}{sup}{\norm{\epsilon R^{1,\epsilon}_{\theta_0}(t)}}
\underset{\alpha\in K_a,k\in\{1,..,n\}}{sup}{\norm{I_p+\Phi_\alpha(t_k,t_{k-1})}\norm{N_k(x_{\alpha_0},\alpha)}}.
  \end{array}$
Let  $\phi(\eta)=sup\{ \abs{\bar{K}_\Delta(\alpha_0,\alpha)-\bar{K}_{\Delta}(\alpha_0,\alpha')}, (\alpha,\alpha') \in
\bar{K}_a^2, \norm{\alpha-\alpha'}\leq \eta\}$, we obtain $w(\bar{U}_{\epsilon,\Delta},\eta)\tend{\epsilon}{0}\phi(\eta)$
under $\P_{\theta_0}$. Assumptions  (S1)-(S4) ensure that $\phi(\eta)\tend{\eta}{0}0$. 
The proof of (i) is achieved using
Theorem 3.2.8 (\cite{dac93}).\\
Consider now the second derivatives of $\bar{U}_{\epsilon}(\alpha,.)$. Noting that, 
\begin{equation*}
\begin{array}{rl}\deriv{^2
N_k(X,\alpha)}{\alpha_i\partial \alpha_j}(\alpha_0)=&{\Delta}\deriv{D_{k,i}(\alpha)}{\alpha_j}(\alpha_0)+\deriv{\Phi_\alpha(t_{k},t_{k-1})}{
\alpha_i}(\alpha_0)\deriv{x_\alpha(t_{k-1}) }{\alpha_j}(\alpha_0)\\ \\
&+\deriv{^2\Phi_\alpha(t_k,t_{k-1})}{\alpha_i\partial \alpha_j}(\alpha_0)\left[X_{t_{k-1}}-x_{\alpha_0}(t_{k-1})\right], \\
\end{array}\end{equation*}
we have 
$$\deriv{^2\bar{U}_{\epsilon}}{\alpha_i\partial \alpha_j}(\alpha_0)=2\epsilon \sqrt{\Delta}E_1+2\Delta E_2,$$ with 
$E_1=\som{k=1}{n} \frac{1}{\Delta}\trans{ \deriv{^2N_{k}(X,\alpha)
}{\alpha_i\partial \alpha_j}}(\alpha_0)\left(\frac{1}{\epsilon\sqrt{\Delta}}N_{k}(X,\alpha_0)\right)$,\\
$E_2=\som{k=1}{n}\left(\frac{1}{\Delta}\trans{\deriv{N_{k}(X,\alpha)}{\alpha_i}}(\alpha_0)\right)\left(\frac{1}{\Delta}
\deriv{N_{k}(X,\alpha)}{\alpha_j}(\alpha_0)\right)$.\\
Using (\ref{limN_k}), (\ref{limdNk}) and that $\frac{1}{\Delta}\trans{ \deriv{^2N_{k}(X,\alpha)
}{\alpha_i\partial \alpha_j}}(\alpha_0)$ is bounded in probability, yields that $E_1$ and $E_2$ are bounded in probability.
Hence, $\deriv{^2\bar{U}_{\epsilon}}{\alpha_i\partial \alpha_j}(\alpha_0)\tend{\epsilon}{0}2\Delta\som{k=1}{n}
\trans{D}_{k,i}(\alpha_0)D_{k,j}(\alpha_0)=2 M_\Delta(\alpha_0)_{i,j}$.\\
The consistency result obtained in (i), and the uniform continuity of $\alpha\rightarrow\Phi_\alpha$ and its derivatives
(see Lemma \ref{rem:Phi}), yields that, under
$\mathbb{P}_{\theta_0}$,
$\underset{t\in[0,1]}{sup}
\norm{\deriv{^2\bar{U}_{\epsilon}}{\alpha^2}(\alpha_0+t(\bar{\alpha}_{\epsilon,\Delta}-\alpha_0))-\deriv{^2\bar{U}_{
\epsilon}}{\alpha^2}(\alpha_0)}\tend{\epsilon}{0}0,$ which completes the proof of (ii).

\subsection{ Proof of Proposition \ref{prop:contrast_betaknown}}
\label{proof:contrast_betaknown}

 The proof of (i) is a repetition of the proof of Proposition \ref{prop:cls}.
The proof of (ii) contains additionnal terms due to the presence of $S_k^{\alpha,f(\alpha)}$ in the contrast process:\\
${\epsilon}^{-1}\deriv{U_{\Delta, \epsilon}}{\alpha_i}(\alpha_0)= T^i_1+ \epsilon \; T^i_2$, with $T^i_1=2\sqrt{\Delta}\som{k=1}{n}\left(\frac{1}{\Delta}\trans{\deriv{N_{k}(X,\alpha)}{\alpha_i}(\alpha_0)}\right)(S_{k}^{\alpha_0,\beta_0})^{-1}\left(\frac{N_{k}(X,\alpha_0)}{\epsilon\sqrt{\Delta}}\right)$,\\$T^i_2=\som{k=1}{n}\trans{\frac{N_{k}(X,\alpha_0)}{\epsilon \sqrt{\Delta}}}\deriv{\left[(S_{k}^{\alpha,f(\alpha)})^{-1}\right]}{\alpha_i}(\alpha_0)\frac{N_{k}(X,\alpha_0)}{\epsilon\sqrt{\Delta}}$.\\
For all i, the term $T^i_2$ is bounded in probability since $S_k^{\alpha,f(\alpha)}$ inherits from $\Phi_\alpha$ its differentiability with respect to $\alpha$ and that $\frac{N_k(X,\alpha)}{\epsilon\sqrt{\Delta}}$ is bounded in probability by \eqref{limN_k}.
Using now (\ref{limN_k}) and (\ref{limdNk}), we obtain, as before, that $\left(T^i_1\right)_{1\leq i\leq a}\tend{\epsilon}{0}\mathcal{N}\left(0,4I_\Delta(\alpha_0,\beta_0)\right)$.\\
$\deriv{^2U_{\Delta,\epsilon}(\alpha_0,\beta_0)}{\alpha_i\partial \alpha_j}=T^{i,j}_1+2\epsilon \sqrt{\Delta}T_2^{i,j}+\epsilon^2 T_3^{i,j}$, where for all $i,j\leq a$:\\
$T_1^{i,j}=2\Delta\som{k=1}{n}\frac{1}{\Delta}\trans{\deriv{N_k(X,\alpha)}{\alpha_i}(\alpha_0)}(S_{k}^{\alpha_0,\beta_0})^{-1}\frac{1}{\Delta}\deriv{N_k(X,\alpha)}{\alpha_j}(\alpha_0) $,\\ $T_2^{i,j}=\som{k=1}{n}\frac{1}{\Delta}\trans{\deriv{^2N_k(X,\alpha)}{\alpha_i\partial \alpha_j}(\alpha_0)}(S_{k}^{\alpha_0,\beta_0})^{-1}\frac{N_k(X,\alpha_0)}{\epsilon\sqrt{\Delta}}+\frac{1}{\Delta}\trans{\deriv{N_k(X,\alpha)}{\alpha_i}(\alpha_0)}\deriv{\left[(S_{k}^{\alpha,f(\alpha)})^{-1}\right]}{\alpha_j}(\alpha_0)\frac{N_{k}(X,\alpha_0)}{\epsilon\sqrt{\Delta}}$,\\  $T_3^{i,j}=\som{k=1}{n}\trans{\frac{N_{k}(X,\alpha_0)}{\epsilon\sqrt{\Delta}}}\deriv{^2\left[(S_{k}^{\alpha,f(\alpha)})^{-1}\right]}{\alpha_i\partial \alpha_j}(\alpha_0)\frac{N_{k}(X,\alpha_0)}{\epsilon\sqrt{\Delta}}.$\\
The two terms $T_2^{i,j}$ and $T_3^{i,j}$ are bounded in probability and therefore \\
\[\deriv{^2U_{\Delta,\epsilon}(\alpha_0,\beta_0)}{\alpha_i\partial \alpha_j}\tend{\epsilon}{0}2I_\Delta(\alpha_0,\beta_0)_{i,j}.\]

\subsection{Proof of Lemma \ref{prop:decompositionGS}}
\label{prooflemme:decompositionGS}
\begin{proof}
Let us study the term $E_k$ defined in Lemma \ref{prop:decompositionGS}. We have $E_k=E_k^1+E_k^2$ with\\ $E_k^1=\integ{t_{k-1}}{t_k}\left(b(\alpha_0,X_{t})-b(\alpha_0,\xzero{})\right)dt+\left(I_p-\fyzero{k}{k-1}\right)\left[X_{t_{k-1}}-\xzero{k-1}\right]$ and\\ $E_k^2=\epsilon\integ{t_{k-1}}{t_k}\left(\sigma(\beta_0,X_s)-\sigma(\beta_0,X_{t_{k-1}})\right)dB_s$. \\
Using that $x\rightarrow b(\alpha,x)$ is Lipschitz, we obtain
{\small\[\begin{array}{ll}\norm{E_k^1}&\leq \Delta C \super{t \in[t_{k-1};t_k]}\norm{X_t-\xzero{}}+\Delta\epsilon\norm{\integ{0}{1}\deriv{b}{x}(\alpha_0,\xzero{})\fyzero{}{k-1}dt \Restun{k-1}}\\
&\leq C' \epsilon\Delta \super{t \in[t_{k-1};t_k]}\norm{\Restun{}}.\end{array}\]}
The proof for $E_k^2$ follows the sketch given in \cite{glo09} (Lemma 1). We prove this result based on the stronger condition $\Sigma$ and $b$ bounded (similarly to  Gloter and S\o rensen in Proposition 1 \cite{glo09} ).\\ We use sequentially Burkh\"{o}lder-Davis-Gundy's inequality and Jensen's inequality to obtain\\
$\begin{array}{lll}\Ec{\norm{E_k^2}^m}{k-1}&\leq& C \epsilon^m\Ec{ \left(\integ{t_{k-1}}{t_k}\norm{\sigma(\beta_0,X_s)-\sigma(\beta_0,X_{t_{k-1}})}^2ds\right)^{m/2}}{k-1}\\
   &\leq&C\epsilon^m\Delta^{m/2-1}\integ{t_{k-1}}{t_k}\Ec{\norm{\sigma(\beta_0,X_s)-\sigma(\beta_0,X_{t_{k-1}})}^m}{k-1}ds.\end{array}$
Then,  using that  $x\rightarrow \sigma(\beta,x)$ is Lipschitz, we obtain:
$\begin{array}{lll}\Ec{\norm{E_k^2}^m}{k-1}&\leq& C'\epsilon^m\Delta^{m/2-1}\integ{t_{k-1}}{t_k}\E{\norm{X_s-X_{t_{k-1}}}^m}ds\\
&\leq& C'\epsilon^m\Delta^{m/2-1}\integ{t_{k-1}}{t_k}\E{\norm{\integ{t_{k-1}}{s}\left(b(\alpha_0,X_u)du+\epsilon\sigma(\beta_0,X_u)dB_u\right)}^m}ds.\end{array}$\\
Since $b$ is  bounded on $\mathcal{U}$, $\norm{\integ{t_{k-1}}{s}b(\alpha_0,X_u)du}\leq K\abs{s-t_{k-1}}$ and the Ito's isometry yields \\
{${\E{\norm{\integ{t_{k-1}}{s}\sigma(\beta_0,X_u)dB_u}^m}\leq\E{\norm{\integ{t_{k-1}}{s}\Sigma(\beta_0,X_u)du}}^{m/2}\leq K \abs{s-t_{k-1}}^{1/2}}.$}\\
Thus, $\Ec{\norm{E_k^2(\beta_0)}^m}{k-1}
\leq C^{''}\epsilon^m\Delta^{m/2-1}\integ{t_{k-1}}{t_k}\abs{s-t_{k-1}}^{m/2}ds\leq C^{(3)}\epsilon^m\Delta^{m}.
$
\end{proof}
The two following results are consequences of Lemma \ref{prop:decompositionGS}. Define, for M a symmetric positive random matrix, using \eqref{def:Nk}:\\ 
 \begin{equation}\label{def:N2k0}
              	N^2_{k,0}(M)=\trans{\Nkzero}M\Sigzero \Nkzero \in \R.
              \end{equation}
Now, for $i=1,2$, if $(M_{k-1}^{(i)})_{k\geq 1}$ is a sequence of $\mathcal{F}_{t_{k-1}}$-measurable symmetric positive matrices  of $M_p(\R)$ satisfying $ \super{k\geq 1}\norm{M_{k-1}^{(i)}}$ is finite in probability,
\begin{equation}\label{prop:moments_Nk2}\frac{1}{\sqrt{n}}\abs{\frac{1}{\epsilon^2\Delta}\som{k=1}{n}\left(\Ec{N_{k,0}^2(M_{k-1}^{(1)})}{k-1}-Tr(M_{k-1}^{(1)})\right)}\tend{\epsilon,\Delta}{0}0\end{equation}
{\small
\begin{equation}\label{prop:moments_Nk4}\abs{\frac{1}{\epsilon^4\Delta^2}\som{k=1}{n}\Ec{N_{k,0}^2(M^{(1)}_{k-1})N_{k,0}^2(M^{(2)}_{k-1})}{k-1}-\left[
 Tr(M^{(1)}_{k-1})Tr(M^{(2)}_{k-1})+2Tr(M^{(1)}_{k-1}M^{(2)}_{k-1})\right]}\tend{\epsilon,\Delta}{0}0\end{equation}}

Indeed, under $\Pzero$, we have \\
{\footnotesize $\begin{array}{lll}\Ec{\trans{\Nkzero}M^{(1)}_{k-1}\Nkzero}{k-1}&=&\som{i,j=1}{p}(M^{(1)}_{k-1})_{i,j}\Ec{\Nkzero_i\Nkzero_j}{k-1}\\
  &=&\som{i,j=1}{p}(M^{(1)}_{k-1})_{i,j}\left(\epsilon^2\Delta\mathbbm{1}_{\{i=j\}}\Sigma(\beta_0,X_{t_{k-1}})_{i,j}+\Ec{(E_k)_i(E_k)_j}{k-1}\right)
 \end{array}$}\\
 where which leads to\\
$\begin{array}{lll}A&=&\frac{1}{\sqrt{n}}\som{k=1}{n}\abs{\frac{1}{\epsilon^2\Delta}\Ec{\trans{\Nkzero}M^{(1)}_{k-1}\Nkzero}{k-1}-Tr(M^{(1)}_{k-1}\Sigma(\beta_0,X_{t_{k-1}}))}\\
&\leq &\frac{C}{\epsilon^2T\sqrt{\Delta}}\supk{M^{(1)}_{k-1}} \som{k=1}{n}\Ec{\norm{E_k}^2}{k-1}\\ 
& \leq &C'\sqrt{\Delta}\end{array}.$\\
The proof of (\ref{prop:moments_Nk4}) is similar and not detailled here.
\subsection{Proof of Lemma \ref{prop:derivatives}-(ii)}
\label{proofii}
Using \eqref{def:Nk}, we have $\frac{1}{\Delta}\deriv{^2N_k(X,\alpha)}{\alpha_i\alpha_j}(\alpha_0)=f_{\Delta}^{(i,j)}(\alpha_0,t_{k-1})+\eta^{(3),(i,j)}_{k-1}(\alpha_0) $
with \\$\eta^{(3),(i,j)}_{k-1}=\frac{1}{\Delta}\deriv{^2\fy{k}{k-1}}{\alpha_i\partial \alpha_j}(\alpha_0)\left[X_{t_{k-1}}-\xzero{k-1}\right]$, and\\
$\begin{array}{lll}f_{\Delta}^{(i,j)}(\alpha_0,t_{k-1})&=&\frac{1}{\Delta}\left(\deriv{^2\x{k}}{\alpha_i\partial\alpha_j}(\alpha_0)-\fyzero{k}{k-1}\deriv{^2\x{k-1}}{\alpha_i\partial\alpha_j}(\alpha_0)\right)\\& &+\frac{1}{\Delta}\left(\deriv{\fy{k}{k-1}}{\alpha_i}(\alpha_0)\deriv{\x{k-1}}{\alpha_j}(\alpha_0)+\deriv{\fy{k}{k-1}}{\alpha_j}(\alpha_0)\deriv{\x{k-1}}{\alpha_i}(\alpha_0)\right).\end{array}$
Using (\ref{bound:dphi/dalpha}), (\ref{bound:d2phi/dalpha2}) and (\ref{bound:d2x/dalpha2}) we obtain that the deterministic quantity $\norm{f_{\Delta}^{(i,j)}}_\infty$ is bounded as $\Delta \rightarrow 0$. Finally, $\eta^{(3),(i,j)}_{k-1}$ is $\Ft{k-1}$-measurable and goes to zero due to Taylor's Stochastic formula as $\epsilon,\Delta \rightarrow 0$.\\

\subsection{Proof of Proposition \ref{borneproba}}
\label{calculation:borneproba}
Let us first recall Lemma 9 in \cite{gen93} that we use in the proof adapted to our notations.
\begin{lemma}
\label{lemma:gc}
	Let $(X_n^k)$ be a $\Ft{k}$-measurable random variable (with $t_{k}=kT/n$), then assume that $\som{k=1}{n}\Ec{X_n^k}{k-1}\rightarrow U$, with U a random variable, and $\som{k=1}{n}\Ec{(X_n^k)^2}{k-1}\rightarrow 0$, then $\som{k=1}{n}X_n^k \rightarrow U$. All the convergences are in probability.
\end{lemma}
 A Taylor expansion with integral remainder, for function $\deriv{\check{U}_{\epsilon,\Delta}}{\alpha_i}$ at point $(\alpha_0,\check{\beta}_{\epsilon,\Delta})$ yields for $i\leq a $
\begin{equation*}-\epsilon \deriv{\check{U}_{\epsilon,\Delta}(\alpha_0,\check{\beta}_{\epsilon,\Delta})}{\alpha_i}=\som{j=1}{a}\left(\integ{0}{1}\epsilon^2\deriv{^2\check{U}_{\epsilon,\Delta}}{\alpha_i \alpha_j}(\alpha_0+ t(\check{\alpha}_{\epsilon,\Delta}-\alpha_0),\check{\beta}_{\epsilon,\Delta})dt\right)_{i,j}\times \epsilon^{-1}(\check{\alpha}_{\epsilon,\Delta}-\alpha_0)_j\end{equation*}
Then, setting \\
$\begin{array}{lll}\check{\eta}(\alpha_0,\check{\beta}_{\epsilon,\Delta})_{i,j}&=&\left(\integ{0}{1}\epsilon^2\deriv{^2\check{U}_{\epsilon,\Delta}}{\alpha_i \alpha_j}(\alpha_0+ t(\check{\alpha}_{\epsilon,\Delta}-\alpha_0),\check{\beta}_{\epsilon,\Delta})-\deriv{^2\check{U}_{\epsilon,\Delta}}{\alpha_i \alpha_j}(\alpha_0,\check{\beta}_{\epsilon,\Delta})dt\right)_{i,j}\\
& &+\epsilon^2\deriv{^2\check{U}_{\epsilon,\Delta}}{\alpha_i \alpha_j}(\alpha_0,\check{\beta}_{\epsilon,\Delta})-2I_b(\alpha_0,,\check{\beta}_{\epsilon,\Delta})_{i,j},                                                                                                                                                                                                                                               \end{array}$\\
with $I_b$ defined in (\ref{Fisher}) we get
\begin{equation}\label{DL:checkalpha} \left[ 2I_b(\alpha_0,\check{\beta}_{\epsilon,\Delta}) + \check{\eta}(\alpha_0,\check{\beta}_{\epsilon,\Delta})\right]\epsilon^{-1}\left( \check{\alpha}_{\epsilon,\Delta}-\alpha_0\right)=-\epsilon\deriv{\check{U}_{\epsilon,\Delta}}{\alpha}(\alpha_0,\check{\beta}_{\epsilon,\Delta}).\end{equation} 
To obtain the tightness of the sequence $\epsilon^{-1}(\check{\alpha}_{\epsilon,\Delta}-\alpha_0)$ w.r.t. $\beta$,  we first study the right handside of (\ref{DL:checkalpha}).\\
Using Lemma \ref{prop:derivatives}-(i), for all $i\in\{1,..,a\}$, $\epsilon \deriv{\check{U}_{\epsilon,\Delta}(\alpha_0,\beta)}{\alpha_i}=\som{k=1}{n}\left(C^i_k+D^i_k\right)$,\\
with $C_k^i=\frac{2}{\epsilon}\trans{\deriv{b}{\alpha_i}(\alpha_0,x_{\alpha_0}(t_{k-1}))}\Sigma^{-1}(\beta,x_{\alpha_0}(t_{k-1}))\Nkzero$ and\\
$\begin{array}{ll}D^i_k=&\frac{2}{\epsilon}\trans{\deriv{b}{\alpha_i}(\alpha_0,x_{\alpha_0}(t_{k-1}))}\left[\Sig-\Sigma^{-1}(\beta,x_{\alpha_0}(t_{k-1}))\right]\Nkzero\\
  &+2\norm{\alpha-\alpha_0}\trans{\eta}_{k}\Sig \Nkzero.
 \end{array}$\\
Set $\tilde{C}_k^i=C_k^i-\Ec{C_k^i}{k-1}$. Let us consider the centered martingale $\som{k=1}{n}\tilde{C}_k^i$. In order to apply  a central limit theorem (see \cite{hal80} Theorem 3.2 p. 58) we have to prove that  $\super{k}\abs{\tilde{C}_k^i}\tend{\epsilon,\Delta}{0}0$,$\som{k=1}{n}\tilde{C}_k^i\tilde{C}_k^j\tend{\epsilon,\Delta}{0}4I_b(\alpha_0,\beta)_{i,j}$ and $\E{\super{k}\left(\tilde{C}_k^i\right)^2}<\infty$ .\\ Note that, since the limit $I_b(\alpha_0,\beta)$ is deterministic, no nesting condition on the $\sigma$-fields is required. \\
Applying Taylor's stochastic formula to $N_k(X,\alpha_0)$ in $\tilde{C}_k^i$ expression yields\\
$ \tilde{C}_k^i=2\trans{\deriv{b}{\alpha_i}(\alpha_0,x_{\alpha_0}(t_{k-1}))}\Sigma^{-1}(\beta,x_{\alpha_0}(t_{k-1}))\left(\sqrt{\Delta}\Zk+\epsilon(\Reste{k}-\Reste{k-1})\right)$. Hence, $\super{k}\abs{C_k^i-\Ec{C_k^i}{k-1}}\tend{\epsilon,\Delta}{0}0$ and $\E{\super{k}\left(C_k^i-\Ec{C_k^i}{k-1}\right)^2}<\infty$. It remains to prove that $\som{k=1}{n}\tilde{C}_k^i\tilde{C}_k^j\tend{\epsilon,\Delta}{0}2I_b(\alpha_0,\beta)_{i,j}$. Let us apply Lemma \ref{lemma:gc} with $X_{n,k}=\tilde{C}_k^i\tilde{C}_k^j$. Then, $\som{k=1}{n}\Ec{\tilde{C}_k^i\tilde{C}_k^j}{k-1}=\frac{1}{\epsilon^2}\som{k=1}{n}\Ec{N_{k,0}^2(M_{k-1})}{k-1}$ where $N^2_{k,0}$ is defined in \eqref{def:N2k0} and $M_{k-1}=4\trans{\deriv{b}{\alpha_i}(\alpha_0,x_{\alpha_0}(t_{k-1}))}\Sigma^{-1}(\beta,x_{\alpha_0}(t_{k-1}))\deriv{b}{\alpha_i}(\alpha_0,x_{\alpha_0}(t_{k-1}))$. Using (\ref{prop:moments_Nk2}) yields that  $\abs{\som{k=1}{n}\Ec{\tilde{C}_k^i\tilde{C}_k^j}{k-1}-4I_b(\alpha_0,\beta)_{i,j}}\tend{\epsilon,\Delta}{0}0$. Moreover (\ref{prop:moments_Nk4}) leads to\\ $\som{k=1}{n}\Ec{(\tilde{C}_k^i\tilde{C}_k^j)^2}{k-1}=\mathcal{O}(\Delta)\rightarrow 0$.\\
Now, we prove that  the centering term $\som{k=1}{n}\Ec{C_k^i}{k-1}\tend{\epsilon,\Delta}{0}0$ and $\som{k=1}{n}D^i_k\tend{\epsilon,\Delta}{0}0$  in probability. For $D_k^i$, \eqref{taylor_stochastique} ensures that $\norm{\frac{1}{\epsilon}\left[\Sig-\Sigma^{-1}(\beta,x_{\alpha_0}(t_{k-1}))\right]}$ is bounded in $\Pz$-probability. Hence, using Lemma \ref{prop:Nk_deltapetit}\\ $V_k^i=\frac{2}{\epsilon}\trans{\deriv{b}{\alpha_i}(\alpha_0,x_{\alpha_0}(t_{k-1}))}\left[\Sig-\Sigma^{-1}(\beta,x_{\alpha_0}(t_{k-1}))\right]+2\norm{\alpha-\alpha_0}\trans{\eta}_{k}\Sig$ is bounded in probability for all k. Since $D^i_k=\trans{V_k^i}N_k(X,\alpha_0)$, \eqref{prop:moments_Nk1} ensures that \\$\som{k=1}{n}D^i_k\tend{\epsilon,\Delta}{0}0$ in $\Pz$-probability.\\
Set $V_{k-1}=\Sigma^{-1}(\beta,x_{\alpha_0}(t_{k-1}))\deriv{b}{\alpha_i}(\alpha_0,x_{\alpha_0}(t_{k-1}))$. \\Then, $\Ec{C^i_k}{k-1}=\frac{1}{\epsilon}\trans{V_{k-1}}\Ec{N_k(X,\alpha_0)}{k-1}$. 
By the Taylor's stochastic formula \begin{equation*}\Nkzero=\epsilon\sqrt{\Delta}\Zk+\epsilon^2\left(\Reste{k}-\fyzero{k}{k-1}\Reste{k-1}\right).\end{equation*}
Using that $Z_k$ is independant from $\Ft{k-1}$,\\
{\small $\Ec{\trans{V_{k-1}}\Nkzero}{k-1}=\epsilon^2\trans{V_{k-1}}\left[ \E{\Reste{k}-\Reste{k-1}}+\Delta\frac{(I_p-\fyzero{k}{k-1})}{\Delta}\Reste{k-1}\right]$.} An Abel transformation to the series yields\\ $\abs{\frac{1}{\epsilon}\som{k=1}{n}\Ec{\trans{V_{k-1}}\Nkzero}{k-1}}\leq T \super{k\in \{ 1,..,n\}}\norm{\frac{V_{k}-V_{k-1}}{\Delta}} \super{t\in[0,T]}\norm{\epsilon\Reste{}}$. Using now that $\supk{\frac{V_k-V_{k-1}}{\Delta}}$ is bounded, we obtain that in probability
\begin{equation}\label{moment_Nk:taylor}
 \abs{\frac{1}{\epsilon}\som{k=1}{n}\trans{V_{k-1}}\Ec{\Nkzero}{k-1}}\tend{\epsilon,\Delta}{0}0.
\end{equation}\\
Combining all these results we get \eqref{a}.
Let us now study $\check{\eta}(\alpha_0,\check{\beta}_{\epsilon,\Delta})_{i,j}$ defined by \eqref{DL:checkalpha}. $\epsilon^2 \deriv{^2\check{U}_{\epsilon,\Delta}}{\alpha_i \alpha_j}(\alpha_0,\beta)=\som{k=1}{n}\left(A^{i,j}_k+B^{i,j}_k\right)$ with\\
 $A^{i,j}_k=2\Delta\trans{\left(\frac{1}{\Delta}\deriv{N_k(X,\alpha)}{\alpha_i}(\alpha_0)\right)}\Sigma^{-1}(\beta,X_{t_{k-1}})\left( \frac{1}{\Delta}\deriv{N_k(X,\alpha)}{\alpha_j}(\alpha_0)\right)$ and \\ $B^{i,j}_k=\trans{\frac{1}{\Delta}\deriv{^2N_k(X,\alpha)}{\alpha_i \alpha_j}(\alpha_0)}\Sigma^{-1}(\beta,X_{t_{k-1}})N_k(X,\alpha_0).$\\
Using that $\Sig$ converges toward $\Sigdo{k-1}$, Lemma \ref{prop:derivatives}-(i) yields that $\som{k=1}{n}A^{i,j}_k \tend{\epsilon,\Delta}{0} 2I_b(\alpha_0,\beta_0)_{i,j}$ (additionnal terms are negligible since they are bounded by $n\Delta\super{k\in\{1,..,n\},\alpha\in K_a}\norm{\eta^{(2)}_{\epsilon,\Delta}}_\infty$). \\
Applying Lemma \ref{prop:derivatives}-(ii) and \eqref{prop:moments_Nk1} yields that $\somn B_k^{i,j}\rightarrow 0$ in $\Pz$-probability. 
Joining all the results we get \eqref{2}.
In addition, since the limit of $\epsilon^2\deriv{^2\check{U}_{\epsilon,\Delta}}{\alpha_i \alpha_j}(\alpha_0,\beta)$ is deterministic, we have \begin{equation}\label{3}\forall t\in[0,1] \super{\beta\in K_b}\norm{\epsilon^2\deriv{^2\check{U}_{\epsilon,\Delta}}{\alpha_i \alpha_j}(\alpha_0+ t(\check{\alpha}_{\epsilon,\Delta}-\alpha_0),\beta)-\epsilon^2\deriv{^2\check{U}_{\epsilon}}{\alpha_i \alpha_j}(\alpha_0,\beta)}\leq K \norm{\check{\alpha}_{\epsilon,\Delta}-\alpha_0}.\end{equation}
Joining (\ref{2}) and (\ref{3}) ensures that $\super{\beta\in K_b}\norm{\check{\eta}(\alpha_0,\beta)}\tend{\epsilon,\Delta}{0}0$.
It remains to prove that $I_b(\alpha_0,\beta)$ is invertible for all $\beta$. According to {\bf(S2)}, $ \Sigma (\beta,x)$ is invertible  $\forall (\beta,x)\in K_b \times U$, which ensures that $\Sigma^{-1}(\beta,x)$ is a coercive bilinear application. The set $K_b$ being compact, the coercive constant can be chosen independently of $\beta$. Using \eqref{Fisher}\\
$\underset{\beta \in K_b}{inf} det( I_b(\alpha_0,\beta)) \geq C \frac{1}{T}\integ{0}{T} \norm{\deriv{b(\alpha_0,x_{\alpha_0}(s))}{\alpha}}^2ds=C_0$, with $C_0$ strictly positive because $I_b(\alpha_0,\beta_0)$ is invertible.\\
Noting $Com(M)$ the comatrix of $M$, we have that $\super{\beta \in K_b} \norm{\trans{Com(I_b(\alpha_0,\beta))}}<\infty$ as a continuous function of $\beta$ and\\
 $\overline{lim}\norm{\epsilon^{-1}(\check{\alpha}_{\epsilon,\Delta}-\alpha_0)}\leq \frac{1}{C_0}\super{\beta \in K_b} \norm{\trans{Com(I_b(\alpha_0,\beta))}} \super{\beta \in K_b}\norm{\frac{1}{\epsilon}\deriv{\check{U}_{\epsilon,\Delta}}{\alpha}(\alpha_0,\beta)} $. Hence $\epsilon^{-1}(\check{\alpha}_{\epsilon,\Delta}-\alpha_0)$ is bounded in $\mathbb{P}_{\theta_0}$-probability, uniformly w.r.t. $\beta$, which achieves the proof of Proposition \ref{borneproba}.\\
\subsection{Proof of Proposition \ref{CV-beta}}
\label{proof:CV-beta}
Using notations \eqref{A1} and \eqref{A2}, we get \\{\small$\frac{1}{n}\left(\check{U}_{\Delta,\epsilon}(\check{\alpha}_{\epsilon,\Delta},\beta) -\check{U}_{\Delta,\epsilon}(\check{\alpha}_{\epsilon,\Delta},\beta_0)\right)=A_1(\beta,\beta_0)+A_2(\alpha_0,\beta,\beta_0)+\left(A_2(\check{\alpha}_{\epsilon,\Delta},\beta,\beta_0)-A_2(\alpha_0,\beta,\beta_0)\right).$}\\
We already obtained the convergence result for $A_1(\beta,\beta_0)$. 
Let us study $A_2(\alpha_0,\beta,\beta_0)$.
Using (\ref{def:N2k0}), $A_2(\alpha_0,\beta,\beta_0)=\frac{1}{\epsilon^2T}\som{k=1}{n}N^2_{k,0}(M_{k-1}(\beta,\beta_0))$, with \\$M_{k-1}(\beta,\beta_0)=\Sig\Sigma(\beta_0,X_{t_{k-1}})-I_p$. Let us now control the conditional moments of $N^2_{k,0}(M_{k-1}(\beta,\beta_0))$. Using (\ref{prop:moments_Nk2}) yields\\
{\small$ \super{\beta\in K_b}\abs{\frac{1}{\epsilon^2T}\som{k=1}{n}\Ec{N_{k,0}^2(M_{k-1}(\beta,\beta_0))}{k-1}-\frac{1}{n}\som{k=1}{n}Tr(M_{k-1}(\beta,\beta_0))}\rightarrow 0$}\\
Hence, $\frac{1}{\epsilon^2T}\som{k=1}{n}\Ec{N_{k,0}^2(M_{k-1}(\beta,\beta_0))}{k-1}\rightarrow \frac{1}{T}\integ{0}{T}Tr(\Sigma^{-1}(\beta,x_{\alpha_0}(t))\Sigma(\beta_0,x_{\alpha_0}(t))dt -p$, uniformly w.r.t. $\beta$. Using (\ref{prop:moments_Nk4}) yields\\
{\small $ \abs{\frac{1}{\epsilon^4T^2}\som{k=1}{n}\Ec{(N_{k,0}^2(M_{k-1}(\beta,\beta_0)))^2}{k-1}-\frac{\Delta}{n}\som{k=1}{n}\left[Tr^2(M_{k-1}(\beta,\beta_0))+2Tr(M_{k-1}^2(\beta,\beta_0))\right]}\rightarrow 0.$}\\
The last term is $\mathcal{O}(\Delta)$ and goes to zero. Applying Lemma \ref{lemma:gc} to \\$X_{n,k}=\frac{1}{\epsilon^2T}N_{k,0}^2(M_{k-1}(\beta,\beta_0))$ yields \\
$A_2(\alpha_0,\beta,\beta_0)\rightarrow\frac{1}{T}\integ{0}{T}Tr(\Sigma(\beta_0,\xzero{})\Sigma^{-1}(\beta,\xzero{})-I_p)dt$ in $\Pzero$-probability. Joining these two results we obtain that $A_1(\beta,\beta_0)+A_2(\alpha_0;\beta,\beta_0)\tend{\epsilon,\Delta}{0} K_2(\alpha_0,\beta_0,\beta)$ uniformly w.r.t. $\beta$.\\
It remains to prove that $A_2(\check{\alpha}_{\epsilon,\Delta},\beta,\beta_0)-A_2(\alpha_0,\beta,\beta_0)\rightarrow 0$ in probability uniformly w.r.t. $\beta$. Then, we use Lemma \ref{prop:Nk_deltapetit} to set $h_k(\alpha,\alpha_0)=\epsilon^{-1}\left(\frac{N_k(x_{\alpha_0},\alpha)}{\Delta}+\norm{\alpha-\alpha_0}\eta_k\right)$. Then,\\
$A_2(\alpha,\beta,\beta_0)-A_2(\alpha_0,\beta,\beta_0)=T_1(\alpha,\alpha_0,\beta)+T_2(\alpha,\alpha_0,\beta)$, with\\
$T_1(\alpha,\alpha_0,\beta)=\frac{\Delta}{n}\som{k=1}{n}\trans{h_k}(\alpha,\alpha_0)\Sig h_k(\alpha,\alpha_0)$ and \\
$T_2(\alpha,\alpha_0,\beta)=(\alpha,\alpha_0,\beta)=\frac{2}{n\epsilon}\som{k=1}{n}\trans{h_k}(\alpha,\alpha_0)\Sig\Nkzero$.\\
By Lemma \ref{prop:Nk-Gamma} and \ref{prop:Nk_deltapetit}, $\supk{h_k(\alpha,\alpha_0)}\leq K\epsilon^{-1}\norm{\alpha-\alpha_0}$ which leads to $\super{\beta\in K_b}\abs{T_1(\check{\alpha}_{\epsilon,\Delta},\alpha_0,\beta)}\leq K \Delta \norm{\epsilon^{-1}\left( \check{\alpha}_{\epsilon,\Delta}-\alpha_0\right)}^2\super{\beta\in K_b}\norm{\Sig}$. Applying Proposition \ref{borneproba} yields that this term goes to zero.\\ $\supk{ \frac{\Delta}{n}\trans{h_k}(\check{\alpha}_{\epsilon,\Delta},\alpha_0)\Sig}\leq \frac{K}{T}\norm{\epsilon^{-1}\left( \check{\alpha}_{\epsilon,\Delta}-\alpha_0\right)}$ is bounded in $\Pz$-probability by Proposition \ref{borneproba}. Finally, applying \eqref{prop:moments_Nk1} ensures that $T_2(\check{\alpha}_{\epsilon,\Delta},\alpha_0,\beta)\rightarrow 0$ and the proof is achieved.
\subsection{ Proof of Theorem \ref{theo:normalite}}
\label{proof:normalite}
The asymptotic normality of $\epsilon^{-1}\left(\check{\alpha}_{\epsilon,\Delta}-\alpha_0\right)$ is obtained just adding the consistency result on $\check{\beta}_{\epsilon,\Delta}$ in the proof of Proposition \ref{borneproba}.\\
Taylor expansion of $\deriv{\check{U}_{\epsilon,\Delta}}{\theta}$ at point $\theta_0=(\alpha_0,\beta_0)$, setting\\ $\theta_t=(\alpha_0+ t(\check{\alpha}_{\epsilon,\Delta}-\alpha_0),\beta_0+ t(\check{\beta}_{\epsilon,\Delta}-\beta_0))$, 
$M_\alpha(\theta)=\left(\epsilon^2 \deriv{\check{U}_{\epsilon,\Delta}}{\alpha_i\partial \alpha_j}(\theta)\right)_{1\leq i,j\leq a}$,\\ $M_{\alpha,\beta}(\theta)=\left(\frac{\epsilon}{\sqrt{n}} \deriv{\check{U}_{\epsilon,\Delta}}{\alpha_i\partial \beta_j}(\theta)\right)_{1\leq i\leq a,1\leq j\leq b}$ and $M_\beta(\theta)=\left(\frac{1}{n} \deriv{\check{U}_{\epsilon,\Delta}}{\beta_i\partial \beta_j}(\theta)\right)_{1\leq i,j\leq b}$ provides\\
{$-\begin{pmatrix}
 	\epsilon \deriv{\check{U}_{\epsilon,\Delta}}{\alpha}(\alpha_0,\beta_0)\\
\frac{1}{\sqrt{n}}\deriv{\check{U}_{\epsilon,\Delta}}{\beta}(\alpha_0,\beta_0)
 \end{pmatrix}=\left(\integ{0}{1}\begin{pmatrix}
M_\alpha & M_{\alpha,\beta}\\
M_{\alpha,\beta} & M_\beta
\end{pmatrix}(\theta_t)dt\right)\begin{pmatrix}
\epsilon^{-1}(\check{\alpha}_{\epsilon,\Delta}-\alpha_0) \\
\sqrt{n}(\check{\beta}_{\epsilon,\Delta}-\beta_0)
\end{pmatrix}$}
   
Let us first study the asymptotic normality of $\beta$.\\
Setting $M_{k-1}^i=\Sigzero\deriv{\Sigma}{\beta_i}(\beta_0,X_{t_{k-1}})$ and noting that\\ $\deriv{log(det(\Sigma(\beta,X_{t_{k-1}})))}{\beta_i}=Tr(M_{k-1}^i)$, we obtain using definition \eqref{def:N2k0},
 $\frac{1}{\sqrt{n}}\deriv{\check{U}_{\Delta,\epsilon}}{\beta_i}(\alpha_0,\beta_0)=\som{k=1}{n}A_k^i$, with 
$A_k^i=\frac{1}{\sqrt{n}}Tr(M_{k-1}^i)-\frac{1}{\epsilon^2\Delta\sqrt{n}} N^2_{k,0}(M_{k-1}^i)$.
Let us first apply Lemma \ref{lemma:gc} with $X_{n,k}=A_k^iA_k^j$.\\
{\small$\begin{array}{rl}\som{k=1}{n}\Ec{A_k^iA_k^j}{k-1}=&\frac{1}{n}\som{k=1}{n}\left(Tr(M_{k-1}^i)Tr(M_{k-1}^j)-\frac{Tr(M^i_{k-1})}{\epsilon^2\Delta}\Ec{N_{k,0}^2(M^j_{k-1})}{k-1}\right)\\
          &+\frac{1}{n}\som{k=1}{n}\left(Tr(M_{k-1}^i)Tr(M_{k-1}^j)-\frac{Tr(M^j_{k-1})}{\epsilon^2\Delta}\Ec{N_{k,0}^2(M^i_{k-1})}{k-1}\right)\\
+\frac{1}{n}\som{k=1}{n}&\left(\frac{1}{\epsilon^4 T \Delta^2}\Ec{N_{k,0}^2(M^i_{k-1})N_{k,0}^2(M^j_{k-1})}{k-1}-Tr(M_{k-1}^i)Tr(M_{k-1}^j)-2Tr(M^{i}_{k-1}M^j_{k-1})\right)\\
  &+\frac{2}{n}\som{k=1}{n}Tr(M^{i}_{k-1}M^j_{k-1}).
         \end{array}$}
Using (\ref{prop:moments_Nk2}) and (\ref{prop:moments_Nk4}), the first three summation terms go to zero, while the last one goes to $4I_\sigma(\alpha_0,\beta_0)_{i,j}$ as a Riemann sum. Using that $\sqrt{n}A_k^i$ is bounded in probability yields that $\norm{A_k^i\tilde{A}_k^j}^2=\mathcal{O}(\frac{1}{n^2})$, leading to $\som{k=1}{n}\Ec{(A_k^iA_k^j)^2}{k-1}\rightarrow 0$. Thus, we obtain $\som{k=1}{n}X_{n,k}\rightarrow 4 I_\sigma(\alpha_0,\beta_0)$. In addition, (\ref{prop:moments_Nk2}) yields\\
{\small$\abs{\som{k=1}{n}\Ec{A_k^i}{k-1}}=\frac{1}{\sqrt{n}}\abs{\frac{1}{\epsilon^2\Delta}\som{k=1}{n}\Ec{N^2_{k,0}(M_{k-1}^i)}{k-1}-Tr(M_{k-1}^i)}\tend{\epsilon,\Delta}{0}0$}\\ and $\super{k\in\{1,..n\}}\abs{\Ec{A_k^i}{k-1}}\rightarrow 0$. Now, setting $\tilde{A}_k^i=A_k^i-\Ec{A_k^i}{k-1}$  leads $\som{k=1}{n}\Ec{\tilde{A}_k^i\tilde{A}_k^j}{k-1}\rightarrow 4 I_\sigma(\alpha_0,\beta_0)$.
Using Taylor's stochastic formula \\
$\begin{array}{ll}\tilde{A}_k^i=&\frac{1}{\sqrt{n}}Tr(M_k^i\left(S_k^{\alpha_0,\beta_0}\Sigma^{-1}(\beta_0,X_{t_{k-1}})-I_p\right)\\ &+\epsilon^2 \trans{(\Reste{k}-\Reste{k-1})}M_k^i\Sigzero(\Reste{k}-\Reste{k-1}).\end{array}$\\
Using \eqref{S_k+1} and \eqref{bound:phi-id} yields that $\Sk=\Sigma(\beta_0,x_{\alpha_0}(t_{k-1})+O(\Delta)$. Hence, $\super{k\in\{1,..n\}}\abs{\tilde{A}_k^i}\tend{\epsilon,\Delta}{0}0$. Using Corollary \ref{Rteps}-(ii) yields $\E{\super{k\in\{1,..n\}}\abs{\tilde{A}_k^i}^2}<\infty$.\\
We can now apply Theorem 3.2 p 58 in \cite{hal80} to the centered martingale $\tilde{A}_k^i$  to obtain the asymptotic normality:
 $\frac{1}{\sqrt{n}}\deriv{\check{U}_{\Delta,\epsilon}}{\beta_i}(\alpha_0,\beta_0)\rightarrow \mathcal{N}\left(0, 4I_\sigma(\alpha_0,\beta_0)\right).$ 

Let us now study the second derivatives of $\check{U}_{\epsilon,\Delta}$, $\frac{1}{n}\deriv{\check{U}_{\Delta,\epsilon}}{\beta_i\partial \beta_j}(\alpha_0,\beta_0)=\som{k=1}{n}B_k^{i,j}$ with\\ $B_k^{i,j}=\frac{1}{n}\left[Tr(L_{k-1}^{i,j})-Tr(M_{k-1}^jM_{k-1}^i)\right]\-\frac{1}{\epsilon^2T}N^2_{k,0}\left(L_{k-1}^{i,j}-M_{k-1}^iM_{k-1}^j-M_{k-1}^jM_{k-1}^i\right),$ and\\
$L_{k-1}^{i,j}=\Sigzero\deriv{\Sigma}{\beta_i\partial \beta_j}(\beta_0,X_{t_{k-1}})$.\\
Using (\ref{prop:moments_Nk2}) we have $\abs{\som{k=1}{n}\Ec{B_k^{i,j}}{k-1}-\frac{1}{n}Tr(M^{i}_{k-1}M^j_{k-1})}\tend{\epsilon,\Delta}{0}0$. Moreover, \\$\norm{B_k^{i,j}}^2=\mathcal{O}(\frac{1}{n^2})$, so $\som{k=1}{n}\Ec{(B_k^{i,j})^2}{k-1}\rightarrow0$. Hence Lemma \ref{lemma:gc} yields\\
$\som{k=1}{n}B_k^{i,j}\tend{\epsilon,\Delta}{0}2I_\sigma(\alpha_0,\beta_0)_{i,j}$.

It remains to study $\frac{\epsilon}{\sqrt{n}}\deriv{\check{U}_{\Delta,\epsilon}}{\alpha_i\beta_j}(\alpha_0,\beta_0)=\som{k=1}{n}C_k^{i,j}$ with,\\
$C_k^{i,j}=\frac{1}{\epsilon\Delta\sqrt{n}}\trans{\deriv{\Nk}{\alpha_i}(\alpha_0)}M_{k-1}^j\Sigzero\Nkzero$.\\
Posing $V_{k-1}^{i,j}=\frac{1}{\sqrt{n}}\Sigzero\trans{M_{k-1}^j}\frac{1}{\Delta}\deriv{\Nk}{\alpha_i}(\alpha_0)$, we have $\supk{\frac{V_{k}^{i,j}-V_{k-1}^{i,j}}{\Delta}}\tend{\epsilon,\Delta}{0}0$ and $\som{k=1}{n}\Ec{C_k^{i,j}}{k-1}=\frac{1}{\epsilon}\som{k=1}{n}\trans{V_{k-1}^{i,j}}\Ec{\Nkzero}{k-1}$. Thus \eqref{moment_Nk:taylor} leads to \\
$\som{k=1}{n}\Ec{C_k^{i,k}}{k-1}\rightarrow 0$. Moreover, setting $M^{i,j}_{k-1}=nV_{k-1}^{i,k}\trans{V_{k-1}^{i,k}}$ we apply (\ref{prop:moments_Nk2}) to obtain $\som{k=1}{n}\Ec{(C_k^{i,k})^2}{k-1}=\frac{1}{n\epsilon^2}\som{k=1}{n}\Ec{N_{k,0}^2(M_{k-1}^{i,j})}{k-1}\rightarrow 0$. Lemma \ref{lemma:gc} leads to\\ $\som{k=1}{n}C_k^{i,j}\tend{\epsilon,\Delta}{0}0$.
The proof is then achieved.

\end{document}